%% file: CharacterizationsArXiv.tex
\documentclass[11pt]{article}

\newcommand{\negeer}[1]{}
\def\notintp{\mathrel{/\kern-0.6em\rhd}} 
\usepackage{amssymb}
\usepackage{amsfonts}
\usepackage{latexsym}
\usepackage{epsfig}
\usepackage{xspace}
\usepackage{color}
\usepackage{amstext}
\usepackage{rotating}
\usepackage{amsmath}
\usepackage{amsthm}
\usepackage{import}
\usepackage{proof}	
\usepackage{epigraph}	
\usepackage[english]{babel}
\usepackage{makeidx}

\renewcommand{\qedsymbol}{$\dashv$}

\newcommand{\inty}{interpretability\xspace}

\newcommand{\pra}{\ensuremath{{\mathrm{PRA}}}\xspace}
\newcommand{\pa}{\ensuremath{{\mathrm{PA}}}\xspace}
\newcommand{\isig}[1]{{\ensuremath {\mathrm{I}\Sigma_{#1}}}\xspace}

\newcommand{\collection}[1]{{\ensuremath{\mathrm{B#1}}\xspace}}
\newcommand{\sonetwo}{{\ensuremath{{\sf S^1_2}}}\xspace}

\newcommand{\zfc}{{\ensuremath{\sf ZFC}}\xspace}

\newcommand{\buss}{\sonetwo}

\newcommand{\formal}[1]{\ensuremath{{\sf {#1}}\xspace}}
\newcommand{\cons}[1]{{\ensuremath{\formal{Con} (#1)\xspace}}}

\newcommand{\cutcons}[2]{{\ensuremath{\formal{Con}^{#2} (#1)\xspace}}}
\newcommand{\pred}{\formal{Pred}}
\newcommand{\axioms}[2]{{\ensuremath{{\formal{Axiom}}_{#1}(#2)\xspace}}}
\newcommand{\bewijs}[2]{{\ensuremath{{\formal{Proof}}_{#1}(#2)\xspace}}}
\newcommand{\freevar}[1]{\formal{FV}{\ensuremath{(#1)}}\xspace}

\newcommand{\restcons}[2]{{\ensuremath{\formal{Con}_{#2}(#1)}}\xspace}
\newcommand{\omeja}[2]{\ensuremath{\omega_1^{#2}(#1)}\xspace}

\newcommand{\translated}[2]{{\ensuremath{{#1}^{#2}}\xspace}}
\newcommand{\trans}[2]{{\ensuremath{{#1}^{#2}}\xspace}}
\newcommand{\transl}[2]{{\ensuremath{{#2}({#1})}\xspace}}
\newcommand{\bigo}[1]{{\ensuremath{\mathcal{O}(#1)}\xspace}}
\newcommand{\concat}[2]{{\ensuremath{#1 \star #2}\xspace}}
\newcommand{\numeral}[1]{{\ensuremath{\overline{#1}}\xspace}}
\newcommand{\pair}[1]{{\ensuremath{\langle #1 \rangle}}\xspace}

\newcommand{\restrict}[2]{{\ensuremath{#1,#2}}\xspace}

\newcommand{\eqbydef}{:=}

\newcommand{\initialsegment}{\preceq_{\sf end}}
\newcommand{\suc}{{\ensuremath{S}}\xspace}

\newcommand{\finapr}[2]{{\ensuremath{#1[#2]}}\xspace}

\newcommand{\expo}{{\ensuremath{{\sf exp}}}\xspace}
\newcommand{\nat}{{\ensuremath{\mathbb{N}}}\xspace}

\theoremstyle{plain}
\newtheorem{theorem}{Theorem}[section]
\newtheorem{lemma}[theorem]{Lemma}

\newtheorem{corollary}[theorem]{Corollary}

\theoremstyle{definition}
\newtheorem{definition}[theorem]{Definition}
\newtheorem{defi}[theorem]{Definition}
\newtheorem{claim}{Claim}

\newtheorem{remark}[theorem]{Remark}

\theoremstyle{remark}

\newcommand{\verz}[1]{\{ #1 \}}
\newcommand{\tupel}[1]{\langle #1 \rangle}

\newcommand{\Index}[1]{#1\index{#1}}
\usepackage{color}

\usepackage{geometry}                
\usepackage{graphicx}
\usepackage{amssymb}
\usepackage{epstopdf}
\begin{document}

\title{Characterizations of interpretability in bounded arithmetic}
\author{Joost J. Joosten\\
University of Barcelona}

\maketitle

\begin{abstract}
This paper deals with three tools to compare proof-theoretic strength of formal arithmetical theories: interpretability, $\Pi^0_1$-conservativity and proving restricted consistency. It is well known that under certain conditions these three notions are equivalent and this equivalence is often referred to as the Orey-H\'ajek characterization of interpretability.

In this paper we look with detail at the Orey-H\'ajek characterization and study what conditions are needed and in what meta-theory the characterizations can be formalized.
\end{abstract}

\section{Introduction}

Interpretations are everywhere used in mathematics and mathematical logic. Basically, a theory $U$ interprets a theory $V$ --we write $U\rhd V$-- whenever there is some translation from the symbols of the language of $V$ to formulas of the language of $U$ so that under a natural extension of this translation the axioms of $V$ are mapped to theorems of $U$. 

The corresponding intuition should be that $U$ is at least as strong or expressible as $V$. And indeed, interpretations are used for example to give relative consistency proofs or to establish undecidability of theories. As such, interpretations are considered an important metamathematical notion. Probably, the first time that interpretations received a formal and systematic treatment has been in the book by A.~Tarski, A.~Mostowski and R.~Robinson (\cite{tars:unde53}). In the current paper we will study that notion of interpretability and also some related notions. Sometimes we speak of \emph{relative} interpretability as to indicate that quantifications become relativized to some domain specifier as we shall define precisely later on. 

We will relate the notion of relative interpretability to two other basic metamathematical notions. The first such notion is the notion of consistency. The notion of consistency is central to mathematical logic and considered key and fundamental. 

A second notion is that of $\Pi_1^0$ conservativity. Below we will exactly define what $\Pi^0_1$ formulas are, but basically, those are formulas in the language of arithmetic which are of the form $\forall x \psi(x)$ where $\psi$ is some decidable predicate. On the other hand, $\Sigma^0_1$ formulas are those of the form $\exists x \psi (x)$ for decidable $\psi$. Since all true theories prove exactly the same set of $\Sigma^0_1$ sentences, the first natural and interesting class of formulas to distinguish theories is on the $\Pi^0_1$ level. Therefore, the notion of $\Pi^0_1$ conservativity has been very central in mathematical logic and foundational discussions. We say that a theory $U$ is $\Pi^0_1$ conservative over $V$ whenever any $\Pi^0_1$ sentence provable by $V$ is also provable by $U$.

The main purpose of this paper is to discuss how these three different notions are related to each other in certain circumstances. This relation is known as the \emph{Orey-H\'ajek} characterization of relative interpretability. 

As such, the paper contains many well-known results and various formulations are taken from \cite{Joosten:2004:InterpretabilityFormalized}. However, we think that it is instructive that all these results are put together and moreover that a clear focus is on the requirements needed so that various implications are formalizable in weak theories. 

Apart from the main focus --which is bringing together facts of the Orey-H\'ajek characterization of relative interpretability and formalizations thereof-- the paper contains a collection of new observations that might come in handy. For example, our simple generalization of Pudl\'ak's lemma as formulated in Lemma \ref{theorem:KeepWitnessSmall} has been a main tool in proving arithmetical correctness of a new series of interpretability principles in \cite{Joosten:2015:TwoSeries}.

\section{Preliminaries}\label{sect:intro}

As mentioned before, a central notion in this paper is that of consistency. Consistency is a notion that concerns syntax: no sequence of symbols that constitute a proof will yield the conclusion that $0=1$. It shall be an important criterion whether or not a theory proves the consistency of another. As such we want that theories can talk about syntax. 

The standard choice to represent syntax is by G\"odel numbering, assigning natural numbers to syntax. Thus, our theories should contain a modicum of arithmetic. In this section we shall make some basic observations on coding and then fix what minimal arithmetic we should have in our base theory. We shall formulate some fundamental properties of this base theory and refer to the literature for further background. Further, we shall fix the notation that is used in the remainder of this paper. 

\subsection{A short word on coding}\label{subs:coding}

Formalization calls for coding of syntax.
At some places in this paper we shall need estimates of codes of
syntactical objects. Therefore it is good to discuss the nature of the
coding process we will employ. However we shall not consider
the implementation details of our coding.

We shall code strings over some finite 
alphabet $A$ with cardinality $a$. A typical coding protocol could be the following. First we define an alphabetic order on
$A$. Next we enumerate all finite strings over $A$ in the following way (pseudo-lexicographic order).

To start, we enumerate all strings of length $0$, then of length $1$, etcetera.
For every $n$, we enumerate the strings of length $n$ in alphabetic order.
The coding of a finite string over $A$ will just be its ordinal number in 
this enumeration. 
We shall now see some easy arithmetical properties of this coding.
We shall often refrain from distinguishing syntactical objects and 
their codes.

\begin{enumerate}

\item
There are $a^n$ many strings of \Index{length} $n$.

\item \label{item:leqcodes}
There are $a^n + a^{n-1} \cdots + 1 = \frac{a^{n+1}-1}{a-1}$ many 
strings of length $\leq n$.

\item \label{item:bigo}
From (\ref{item:leqcodes}) it follows that the code of
a syntactical object of length $n$, is 
$\bigo{\frac{a^{n+1}-1}{a-1}} = \bigo{a^n}$ big.

\item
Conversely, the length of a syntactical object that has code $\varphi$
is  \bigo{ | \varphi |} (logarithm/length of $\varphi$) big.

\item
If $\varphi$ and $\psi$ are codes of syntactical objects, the \Index{concatenation}
\concat{\varphi}{\psi} of $\varphi$ and $\psi$ is  
\bigo{\varphi \cdot \psi} big. For, 
$|\concat{\varphi}{\psi}|=|\varphi| + |\psi|$, whence by (\ref{item:bigo}), 
$\concat{\varphi}{\psi} \approx a^{|\varphi| + |\psi|}=a^{|\varphi |} \cdot
a^{|\psi|} = \varphi \cdot \psi$.

\item \label{item:substitution}
If $\varphi$ and $t$ are (codes of) syntactical objects, then 
$\varphi_x (t)$ is  \bigo{\varphi^{|t|}} big. Here 
$\varphi_x (t)$ denotes the syntactical object that results from 
$\varphi$ by replacing every (unbounded) occurrence of $x$ by $t$.
The length of $\varphi$ is about $|\varphi |$. In the worst case,
these are all $x$-symbols. In this case, the length of 
$\varphi_x (t)$ is $|\varphi | \cdot |t |$ and thus
$\varphi_x (t)$ is  
$\bigo{a^{|\varphi | \cdot |t |}}= \bigo{t^{|\varphi |}}=
\bigo{\varphi^{|t|}} = \bigo{2^{|\varphi |\cdot |t|}}$ big.

\end{enumerate}

As mentioned, we shall refrain from the technical characteristics of our coding and refer to the literature for examples. Rather, we shall keep in mind restrictions on the sizes and bounds as mentioned above. Also, we shall assume that we work with a natural poly-time coding with poly-time decoding functions so that the code of substrings is always smaller than the code of the entire string.

\subsection{Arithmetical theories}

Since substitution is key to manipulating syntax we need, by our observations above, a function whose growth-rate can capture substitution. Thus, we choose to work with the smash function $\sharp$ defined by $x\sharp y := 2^{|x|\cdot |y|}$ where $|x| := \lceil \log_2 (x+1)\rceil$ is the length of the number $x$ in binary. We shall often also employ the function $\omega_1$ which is of similar growth-rate and defined by $\omega_1(x):= 2^{|x|^2}$.

Next, we need a certain amount of induction. For a formula $\varphi$, the regular induction formula $I_\varphi$ is given by
\[
\varphi (0) \ \wedge \  \forall x\ (\varphi (x) \to \varphi (x+1)) \ \to \ \forall x \ \varphi (x).
\]
However, it turns out that we can work with a weaker version of induction called polynomial induction denoted by PIND:
\[
\varphi (0) \wedge \forall x\ (\varphi (\lfloor \frac{1}{2}x \rfloor)  \to \varphi (x)) \ \to \forall x\ \varphi (x)
\]
or equivalently
\[
\varphi (0) \wedge \forall x\ (\varphi (x ) \to \varphi (2x)) \wedge \forall x\ (\varphi (x ) \to \varphi (2x{+}1)) \ \to \forall x\ \varphi (x).
\]
The idea is that one can conclude $\varphi(x)$ by only logarithmically many calls upon the induction hypothesis with this PIND principle. For example to conclude $\varphi(18)$ we'ld go $\varphi(0)\to\varphi(1)\to\varphi(2)\to\varphi(4)\to\varphi(9)\to\varphi(18)$. 

Typically, induction on syntax is of this nature and in order to conclude a property of (the G\"odel number of) some formula $\psi$ we need to apply the induction hypothesis to the number of subformulas of $\psi$ which is linear in the length of $\psi$. Thus, most inductions over syntax can be established by PIND rather than the regular induction schema.

Moreover, we shall restrict the formulas on which we allow ourselves to apply  PIND to so to end up with a weak base theory. As we shall see, most of our arguments can be formalized within Buss' theory\footnote{As mentioned, the 
substitution operation on codes of syntactical objects asks
for a function of growth rate $x^{|x|}$. In Buss's \sonetwo this is the \Index{smash function} \label{symb:sharp}$\sharp$.
In the theory $I\Delta_0+\Omega_1$ this is the 
function \omeja{x}{}. However, contrary to \sonetwo, the theory $I\Delta_0+\Omega_1$ --aka ${\sf S}_2$-- is not known to be finitely axiomatizable.} 
 \sonetwo.

The theory \sonetwo is formulated in the language of arithmetic $\{ 0, S, +, \cdot, \sharp, |x|, \lfloor \frac{1}{2}x\rfloor , \leq \}$. Apart from some basic axioms that define the symbols in the language, \sonetwo is axiomatized by PIND induction for $\Sigma_1^b$ formulas. The $\Sigma_1^b$ formulas are those formed from atomic formulas via the boolean operators, sharply bounded quantification and bounded existential quantification. Sharply bounded quantification is quantification of the from ${\mathcal Q} \ x{<}|t|$ for ${\mathcal Q}\in \{ \forall, \exists \}$. Bounded existential quantification in contrast, is of the form $\exists \, x{<}t$. We refer the reader for \cite{BusH2} or \cite{HP} for further details and  for the definitions of the related $\Sigma^b_n$ and $\Pi^b_n$ hierarchies.

Equivalent to the PIND principle (see \cite[Lemma 5.2.5]{Krajicek:1995:BoundedArithmeticPropLogicComplTheory}) is the length induction principle LIND:
\[
\varphi (0) \wedge \forall x\ (\varphi (x ) \to \varphi (x{+}1)) \ \to \forall x\ \varphi (| x|).
\]
So, from the progressiveness of $\varphi$, we can conclude $\varphi(x)$ for any $x$ for which the exponentiation is defined. We shall later see that if we are working with definable cuts (definable initial segments of the natural numbers with some natural closure properties) we can without loss of generality assume that exponentiation is defined for elements of this cut.

Although most of our reasoning can be performed in \sonetwo, we sometimes mention stronger theories. As always Peano Arithmetic (\pa) contains open axioms that define the symbols $0, S, +$ and $\cdot$ and induction axioms $I_\varphi$ for any arithmetical formula $\varphi$. Similarly, \isig{n} is as \pa where instead we only have induction axioms $I_\varphi$ for $\varphi \in \Sigma_n$. Here, $\Sigma_n$ refers to the usual arithmetical hierarchy (see e.g.~\cite{HP}) in that such formulas are written as a decidable formula preceded by a string of $n$ alternating quantifiers with an existential quantifier up front. In case no free variables are allowed in the induction formulas, we flag this by a superscript ``$-$" as in $\mathrm{I}\Sigma_{n}^{-}$.

Another important arithmetical principle that we will encounter frequently is collection. \label{symb:collection}
For example \collection{\Sigma_n} is the so-called \Index{collection} scheme for 
$\Sigma_n$-formulae. Roughly, \collection{\Sigma_n} says that the range of a 
$\Sigma_n$-definable function on a finite interval is again finite. A mathematical 
formulation is 
$\forall\,  x{\leq}u\, \exists y\ \sigma (x,y) \rightarrow 
\exists z\, \forall\,  x{\leq}u\, \exists \, y{\leq}z\ \sigma (x,y)$ where
$\sigma (x,y) \in \Sigma_n$ may contain other variables too. 

The least number principle $\mathrm{L}\Gamma$ for a class of formulas is the collection $\exists x\ \varphi (x) \to \exists x\ (\varphi (x) \wedge \forall \, y{<}x\ \neg \varphi (y))$ for $\varphi \in \Gamma$.

\subsection{Numberized theories}

The notion of interpretability applies to any pair of theories and not necessarily need they contain any arithmetic. However, in this paper we will prove that $U\rhd V$ can in various occasions be equivalent to other properties that are stated in terms of numbers. For example, in certain situations we have that $U\rhd V$ is equivalent to $U$ proving all the $\Pi^0_1$ formulas that $V$ does. Clearly, in this situation we should understand that $U$ and $V$ come with a natural interpretation of numbers.

\begin{definition}
We will call a pair $\tupel{U,k}$\/  a {\em \Index{numberized theory}}\/ 
if 
$k:U\rhd {\sf S}^1_2$.
A theory $U$ is {\em numberizable} or {\em arithmetical} if for some $j$, 
\pair{U,j} is a numberized theory.
\end{definition}

From now on, we shall only consider numberizable or numberized theories.
Often however, we will fix a numberization $j$ and reason about the
theory \pair{U,j} as if it were formulated in the language of arithmetic.

A disadvantage of doing so is clearly that our statements may
be somehow misleading; when we think of, e.g.,  \zfc we do not 
like to think of it as coming with a fixed numberization. However, for the kind of characterizations treated in this paper, it is really needed to have numbers around. We shall most of the times work with 
\index{sequential theory} sequential theories. Basically, sequentiality means 
that any finite sequence of objects can be coded.

\subsection{Metamathematics in numberized theories}
On many occasions, we want to represent numbers by terms (numerals) and 
then consider the code of that term. It is  not a good idea to represent
a number $n$ by 
\[
\overbrace{\suc \ldots \suc}^{n \mbox{ times}}0.
\]
For, the
length of this object is $n+1$ whence its code is about $2^{n+1}$ and
we would like to avoid the use of exponentiation. In the setting of weaker arithmetics
it is common practice to use so-called \emph{\Index{efficient numerals}}. These numerals
are defined by recursion as follows. $\numeral{0}=0$; 
$\numeral{2{\cdot} n}= (\suc \suc 0) \cdot \numeral{n}$ and 
$\numeral{2{\cdot}n +1}= \suc ((\suc \suc 0) \cdot \numeral{n})$. Clearly, these numerals implement the system of dyadic notation which perfectly ties up with the PIND principle. Often we shall refrain between distinguishing $n$ from its numeral $\overline n$ or even the G\"odel number $\ulcorner \overline n \urcorner$ of its numeral.

As we want to do arithmetization of syntax, our theories should be coded in 
a simple way. We will assume that all our theories $U$ have an
axiom set that is decidable in polynomial time.
%
%
That is, there is some formula \axioms{U}{x} which is
$\Delta_1^b$ (both the formula and its negation are provably equivalent to a $\Sigma^b_1$ formula) in \sonetwo, with 
\begin{enumerate}
\item[]
$\sonetwo \vdash \axioms{U}{\varphi}$ iff $\varphi$ is an axiom of $U$.
%
\end{enumerate}
The choice of $\Delta_1^b$-axiomatizations is also
motivated by Lemma \ref{lemm:boundedsigmacompleteness} below. Most natural theories like \zfc or \pa indeed have $\Delta_1^b$-axiomatizations. Moreover, by a sharpening of Craig's trick, any recursive theory is deductively equivalent to one with a $\Delta_1^b$-axiomatization.

We shall employ the standard techniques and concepts necessary for the arithmetization of syntax. Thus, we shall work with provability predicates $\Box_U$ corresponding uniformly to arithmetical theories $U$. We shall adhere to the standard dot notation so that, for example, $\Box_U \varphi (\dot x)$ denotes a formula with one free variable $x$ so that for each value of $x$, $\Box_U \varphi (\dot x)$ is provably equivalent to $\Box_U \varphi (\overline x)$.

We shall always write the formalized version of
a concept in sans-serif style. For example, \bewijs{U}{p,\varphi} stands for the
formalization of ``$p$ is a $U$-proof of $\varphi$'', \cons{U} stands for 
the formalization of ``$U$ is a consistent theory'' and so forth. It is known that for theories $U$ with a poly-time axiom set, the formula \bewijs{U}{p,\varphi} can be taken to be in $\Delta^b_1$ being a poly-time decidable predicate. Again, \cite{BusH2} and \cite{HP} are adequate references.

For already really weak theories $T$ we have $\Sigma_1$-completeness in the sense that $T$ proves any true $\Sigma_1$ sentence.
However, proofs of $\Sigma_1$-sentences $\sigma$ are multi-exponentially
big, that is, $2^{\sigma}_n$ for some $n$ depending on $\sigma$. 
(See e.g., \ \cite{HP}.) As such, we cannot expect that we can formalize the $\Sigma_1$ completeness theorem in theories where exponentiation is not necessarily total.

However, for $\exists \Sigma_1^b$-formulas we do have a completeness 
theorem (see \cite{BusH2}) in bounded arithmetic.
\label{symb:forallform}
From now on, we shall often write a sup-index to a quantifier to
specify the domain of quantification.

\begin{lemma}\label{lemm:boundedsigmacompleteness}
If $\alpha (x)\in \exists \Sigma_1^b$, then there 
is some standard natural number $n$ such that
\[
\sonetwo \vdash \forall x\ 
[\alpha (x) \rightarrow \exists \, p{<}\omeja{x}{n} \ \bewijs{U}{p, \alpha (\dot{x})}] .
\]
This holds for any reasonable arithmetical theory $U$.
Moreover, we have also a formalized version of this statement.
\[
\sonetwo \vdash \forall^{\exists \Sigma_1^b} \alpha \, \exists n\  
\Box_{\sonetwo}( \forall x\ 
[\dot{\alpha} (x) \rightarrow \exists \, p{<}\omeja{x}{\dot n} \ \bewijs{U}{p, \dot \alpha (\dot{x})}]) .
\]
\end{lemma}

\subsection{\index{reflexive theory} Consistency and reflexive theories}

Since G\"odel's second incompleteness theorem, we know that no recursive theory that is consistent can prove its own consistency. For a large class of natural theories we do have a good approximation of proving consistency though. A theory is \emph{reflexive} if it proves the consistency of all of its finite subtheories. Reflexivity is a natural notion and most natural non-finitely axiomatized theories are reflexive like, for example, primitive recursive arithmetic and \pa.

Many meta-mathematical statements involve the notion of 
reflexivity. There exist various ways in which reflexivity can be formalized, and 
throughout the literature we can find many different 
formalizations. For stronger theories, all these 
formalizations coincide. But for weaker theories, the differences are essential.
We give some formalizations of reflexivity.

\begin{enumerate}
\item
$\forall n\ U \vdash \cons{\finapr{U}{n}}$
where \label{symb:finapr} \finapr{U}{n} denotes the conjunction of the first $n$ axioms 
of $U$.

\item
$\forall n\ U \vdash \cons{{U}{\upharpoonright}{n}}$ where 
\label{symb:restrictedtheory}$\cons{{U}{\upharpoonright}{n}}$
denotes that there is no proof of
falsity using only axioms of $U$ with G\"odel numbers $\leq n$.

\item
$\forall n\ U \vdash \restcons{U}{n}$ where \label{symb:restcons}\restcons{U}{n} 
denotes
that there is no proof of falsity with a proof $p$ where 
$p$ has the following properties. All non-logical axioms of $U$ that
occur in $p$ have G\"odel numbers $\leq n$. All formulas
$\varphi$ that occur in $p$ have a logical complexity 
\label{symb:rho}$\rho (\varphi) \leq n$.\\
Here $\rho$ is some 
complexity measure that basically counts the number 
of quantifier alternations in $\varphi$. Important features of this
$\rho$ are that for every $n$, there are truth predicates
for formulas with complexity $n$. Moreover, the $\rho$-measure of a formula
should be more or less (modulo some poly-time difference, 
see Remark \ref{rema:translation}) preserved under
translations. An example of such a $\rho$ is given in \cite{viss:unpr93}.
\end{enumerate}

It is clear that $(2) \Rightarrow (3)$ can be proven in any weak base theory. For the
corresponding provability notions, the implication reverses. In this
paper, our notion of reflexivity shall be the third one.

We shall write $\Box_{U,n} \varphi$ for 
$\neg \restcons{U+\neg \varphi}{n}$ or, equivalently, 
$\exists p\ \bewijs{U,n}{p,\varphi}$. Here,  
$\bewijs{U,n}{p,\varphi}$ denotes that $p$ is a 
$U$-proof of $\varphi$ with all axioms in $p$ are $\leq n$ and for
all formulas $\psi$ that occur in $p$, we have $\rho(\psi)\leq n$.

\begin{remark}\label{rema:sigma}
An inspection of the proof of provable 
$\Sigma_1$-completeness (Lemma \ref{lemm:boundedsigmacompleteness})
gives us some more information. The proof $p$ that witnesses the 
provability in $U$ of some $\exists \Sigma^b_1$-sentence 
$\alpha$, can easily be taken so that
all axioms occurring in $p$ are about as big and complex as $\alpha$.
Thus, from $\alpha$, we get for some $n$ (depending linearly 
on $\alpha$) that \bewijs{U,n}{p,\alpha}.
\end{remark}

If we wish to emphasize the fact that our theories are not necessarily in 
the language of arithmetic, but just can be numberized, our formulations
of reflexivity should be slightly changed. For example, 
$(3)$ will for some \pair{U,j} look like 
$j: U \rhd \sonetwo + \verz{\restcons{U}{n} \mid n\in \omega}$.

If $U$ is a reflexive theory, we do not necessarily have any 
reflection principles.
That is, we do not have
$U \vdash  \Box_V \varphi \rightarrow \varphi$ for some natural $V\subset U$ and
for some natural class of formulae $\varphi$.
We do have, however, a weak form of $\forall \Pi_1^b$-reflection.
This is expressed in the following lemma.

\begin{lemma}\label{lemm:weakpireflexion}
Let $U$ be a reflexive theory. Then
\[
\sonetwo \vdash \forall^{\forall \Pi_1^b} \pi \, \forall n \ \Box_U 
\forall x \ (\Box_{\restrict{U}{n}}\pi (\dot{x}) \rightarrow \pi (x)) .
\]
\end{lemma}

\begin{proof}
Reason in 
\sonetwo
and fix $\pi$ and $n$. Let $m$ be such
that we have 
(see Lemma
\ref{lemm:boundedsigmacompleteness} and
Remark \ref{rema:sigma})
\[
\Box_U \forall x \ (\neg \pi (x) \rightarrow \Box_{\restrict{U}{m}}\neg
\pi (\dot{x})) .
\]
Furthermore, let $k \eqbydef \max \{n,m\}$. Now, reason in $U$, fix some 
$x$ and assume $\Box_{\restrict{U}{n}}\pi (x)$. Thus, clearly also
$\Box_{\restrict{U}{k}}\pi (x)$. If now $\neg \pi (x)$, then also 
$\Box_{\restrict{U}{k}}\neg \pi (x)$, whence $\Box_{\restrict{U}{k}}\bot$.
This contradicts the reflexivity, whence $\pi (x)$. As $x$ was arbitrary we 
get $\forall x \ (\Box_{\restrict{U}{n}}\pi (x) \rightarrow \pi (x))$.
\end{proof}
We note that this lemma also holds for the other notions of restricted 
provability we introduced in this subsection.

\section{Formalized interpretability}

As we already mentioned, our notion of \inty is 
the one
studied by Tarski et al in \cite{tars:unde53}. In that notion, any axiom needs to be provable after translation. Under some fairly weak conditions this implies that also theorems are translated to theorems. However, in the domain of bounded arithmetics we do not generally have this. In the realm of formalized interpretation therefore, there has been a tendency to consider a small adaptation of the original notion Tarski. This adaptation as introduced by Visser is called \emph{smooth} interpretability. In this subsection we shall exactly define this notion and see how it relates to other notions of formalized interpretability. In various ways, one can hold that theorems interpretability as discussed below is actually the more natural formalized version of interpretability.

The theories that we study in this paper are theories formulated
in first order predicate logic. All theories have a finite signature 
that contains identity. For simplicity we shall assume that all our
theories are formulated in a purely relational way.
Here is the formal definition of a relative interpretation.

\begin{defi}\label{defi:interpretation}
A \emph{translation} $k$  of the language of a theory $S$ into the language of a theory $T$
is a pair $\langle \delta, F \rangle$ for which the 
following holds. 

The first component $\delta$, is called the \emph{domain specifier} and is a formula in the
language of $T$ with a single free variable. This formula is used to specify the 
domain of our interpretation.

The second component, $F$, is a finite map that sends relation symbols $R$ (including identity) from the language of $S$, to formulas $F (R )$ in the language of $T$. We demand for all $R$ that the number of free variables of $F (R )$ equals the arity of $R$.\footnote{Formally,
we should be more precise and specify our variables.} 
Recursively we define the \Index{translation} \trans{\varphi}{k}
\label{symb:transfi}
of a formula 
$\varphi$ in the language of $S$ as follows.
\begin{itemize}
\item
$\trans{(R (\vec{x}))}{k} = F(R)(\vec{x})$;

\item
$\trans{(\varphi \wedge \psi)}{k}= \trans{\varphi}{k} \wedge 
\trans{\psi}{k}$ and likewise for other boolean connectives;\\ 
(in particular, this implies $\trans{\bot}{k}=\bot$);

\item
$\trans{(\forall x\ \varphi (x))}{k}= \forall x\ (\delta (x) 
\rightarrow \trans{\varphi}{k})$ and analogously for the existential
quantifier.
\end{itemize}
A relative interpretation $k$  of a theory $S$ into a theory $T$
is a translation $\langle \delta, F \rangle$ so that $T \vdash \trans{\varphi}{k}$ for all axioms $\varphi$ of $S$.
\end{defi}

To formalize insights about interpretability in weak meta-theories like ${\sf S}^1_2$
we need to be very careful. Definitions of interpretability that are unproblematically
equivalent in a strong theory like, say, $\isig{1}$ diverge in weak theories. 
As we shall see, the major
source of problems is the absence of $\collection{\Sigma_1}$.

In this subsection, we 
study various divergent definitions of interpretability. We start by making an elementary observation on interpretations. Basically,
the next definition and lemma say that translations 
transform proofs into translated proofs.

\begin{definition}\label{defi:prooftransl}
Let $k$ be a \Index{translation}.
By recursion on a proof $p$ in natural deduction we define 
the translation of $p$ under $k$, we write
\label{symb:prooftrans}
\translated{p}{k}. 
For this purpose, we first define \transl{\varphi}{k} for  
formulae $\varphi$ to be\footnote{To be really precise we should
say that, for example, we let smaller $x_i$ come first in 
$\bigwedge_{x_i\in \freevar{\varphi}}\delta (x_i)$.} 
$\bigwedge_{x_i\in \freevar{\varphi}}\delta (x_i) \rightarrow \translated{\varphi}{k}$.
Here \freevar{\varphi} denotes the set of free variables of $\varphi$.
Clearly, this set cannot contain more than $|\varphi|$ elements, whence
\transl{\varphi}{k} will not be too big. Obviously, for sentences
$\varphi$, we have $\transl{\varphi}{k}=\translated{\varphi}{k}$.

If $p$ is just a single assumption $\varphi$, then 
\trans{p}{k} is \transl{\varphi}{k}. The translation 
of the proof constructions are defined precisely in such a way that
we can prove Lemma \ref{lemm:translatedproofs} below. For example, the translation of
$$
\infer{\varphi\wedge\psi}{\varphi & \psi}
$$ 
will be
$$
\infer[\rightarrow I, 1]{\bigwedge_{x_i\in \freevar{\varphi \wedge \psi}}\delta (x_i)
\rightarrow \translated{\varphi}{k} \wedge \translated{\psi}{k}}
{
\infer{\translated{\varphi}{k} \wedge \translated{\psi}{k}}
{\infer{\translated{\varphi}{k}}
{
\infer{\bigwedge_{x_i\in \freevar{\varphi}}\delta (x_i)}
{[\bigwedge_{x_i\in \freevar{\varphi \wedge \psi}}\delta (x_i)]_1} &
\bigwedge_{x_i\in \freevar{\varphi}}\delta (x_i) \rightarrow \translated{\varphi}{k}
} 
&
\infer{\translated{\psi}{k}}
{\mathcal{D}}
}
}
$$
%
%
where $\mathcal{D}$ is just a symmetric copy of the part above \translated{\varphi}{k}.
We note that the translation of the proof constructions
is available\footnote{More efficient translations on proofs are
also available. However they are less uniform.} 
in \sonetwo, as the number of free variables in 
$\varphi \wedge \psi$ is bounded by $|\varphi \wedge \psi|$.
\end{definition}

\begin{lemma}\label{lemm:translatedproofs}
If $p$ is a proof of a sentence $\varphi$ with assumptions in some set of sentences
$\Gamma$, then for any translation $k$, \translated{p}{k} is a proof
of \translated{\varphi}{k} with assumptions in \translated{\Gamma}{k}.
\end{lemma}

\begin{proof}
Note that the restriction on sentences is needed. 
For example
$$
\infer{\psi (x)}{\forall x \  \varphi (x) & \forall x\ (\varphi (x) \rightarrow 
\psi (x))}
$$
but
$$
\infer{\delta (x) \rightarrow \psi^k (x)}
{(\forall x \  \varphi (x))^k & (\forall x\ (\varphi (x) \rightarrow 
\psi (x)))^k}
$$
and in general $\nvdash (\delta (x) \rightarrow \translated{\psi}{k})\leftrightarrow 
\translated{\psi}{k}$. The lemma is proved by induction on $p$. To account for 
formulas in the induction, 
we use the notion \transl{\varphi}{k} from Definition \ref{defi:prooftransl},
which
is tailored precisely to let the induction go through.
\end{proof}

\begin{remark}\label{rema:translation}
The proof translation leaves all the structure invariant. Thus,
there is a provably total (in \sonetwo) function $f$ such that,
if $p$ is a $U,n$-proof of $\varphi$, then $\trans{p}{k}$
is a proof of $\trans{\varphi}{k}$, where $\trans{p}{k}$ has the 
following properties. All axioms in \trans{p}{k} are 
$\leq f(n,k)$ and all formulas $\psi$ in \trans{p}{k}
have $\rho(\psi)\leq f(n,k)$.
\end{remark}


\begin{figure}
\begin{center}
\input{transitive.pstex_t}
\end{center}
\caption{Transitivity of interpretability}\label{pict:trans}
\end{figure}
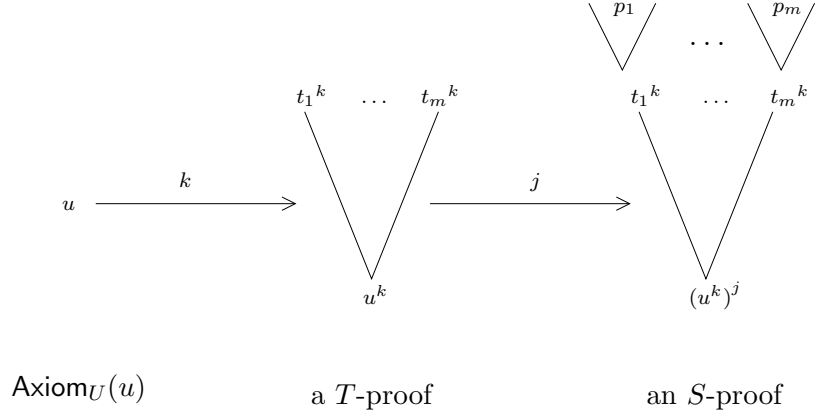

There are various reasons to give, why we want the notion of interpretability
to be provably transitive, that is, provably $S\rhd U$ whenever both $S \rhd T$ and $T\rhd U$.
The obvious way of proving this would be by composing (doing the one after the other)
two interpretations. 
Thus, if we have
$j: S \rhd T$ and $k: T\rhd U$ we would like to have
$j\circ k : S\rhd U$ where $j\circ k$ denotes a natural composition of translations.

If we try to perform a proof as depicted in 
Figure \ref{pict:trans}, at
a certain point we would like to collect the $S$-proofs $p_1,\cdots ,p_m$
of the $j$-translated 
$T$-axioms used in a proof of a $k$-translation of an axiom $u$ of $U$, and take 
the maximum of all such proofs. But to see that such a maximum exists, 
we precisely need $\Sigma_1$-collection.

However, it is desirable to also reason about \inty in 
the absence of \collection{\Sigma_1}. A trick is needed to circumvent 
the problem of the unprovability of transitivity (and many other 
elementary desiderata).

One way to solve the problem is by switching to a notion
of \inty where the needed collection has been built in. This is
the notion of smooth (axioms) \inty as in Definition \ref{defi:intynotions}.
In this paper we shall mean by \inty, unless mentioned otherwise, always
smooth \inty.
In the presence of \collection{\Sigma_1} this notion will coincide with
the earlier defined notion of \inty, as Theorem \ref{theo:intynotions} 
tells us.

%
%
%
%

\begin{definition}\label{defi:intynotions}
We define the notions of \index{interpretability!axioms}
\index{interpretability!smooth axioms}\index{interpretability!theorems}axioms
\index{interpretability!smooth theorems}interpretability 
$\rhd_a$, theorems interpretability 
$\rhd_t$, 
smooth axioms interpretability $\rhd_{sa}$ and
smooth theorems \inty $\rhd_{st}$. 
\[
\begin{array}{lll}
j : U\rhd_a V & \eqbydef & \forall v\, \exists p \ 
(\axioms{V}{v} \rightarrow \bewijs{U}{p,\translated{v}{j}}) \\
j : U\rhd_t V & \eqbydef & \forall \varphi \, \forall p \, \exists p'\ 
(\bewijs{V}{p,\varphi} \rightarrow \bewijs{U}{p',\translated{\varphi}{j}})\\
j : U\rhd_{sa} V & \eqbydef & 
\forall x \, \exists y\, \forall \, v{\leq} x \, \exists\, p{\leq}y\ 
(\axioms{V}{v}\rightarrow \bewijs{U}{p,\translated{v}{j}})\\
j : U\rhd_{st} V & \eqbydef & 
\forall  x\, \exists y\, \forall \, \varphi {\leq} x\, \forall \, p{\leq}x \, \exists\, p'{\leq}y\ 
(\bewijs{V}{p,\varphi} \rightarrow \bewijs{U}{p',\translated{\varphi}{j}})\\
\end{array} 
\]
\end{definition}

It is now easy to see that $\rhd_a$ is indeed provably transitive over very weak base theories. For $\rhd_t$ this follows almost directly from the definition.

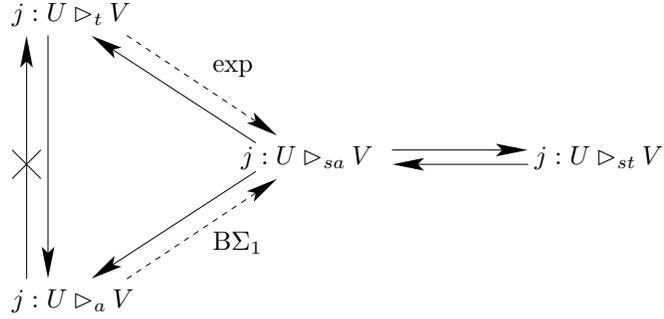
\begin{figure}
\begin{center}
\input{smoothnew.pstex_t}
\end{center}
\caption{Versions of relative interpretability. The dotted arrows indicate that an additional condition is needed in our proof; the condition written next to it. The arrow with a cross through it,
indicates that we know that the implication fails in \sonetwo.}\label{pict:smooth}
\end{figure}

\begin{theorem}\label{theo:intynotions}
In \sonetwo we have all the arrows as depicted in Figure \ref{pict:smooth}. 
\end{theorem}

\begin{proof}
We shall only comment on the arrows that are not completely trivial.\\

$\bullet$ $T\vdash j: U\rhd_aV \rightarrow j:U\rhd_{sa}V$, if $T\vdash \collection{\Sigma_1}$.
So, reason in $T$ and suppose 
$\forall v\, \exists p \ 
(\axioms{V}{v} \rightarrow \bewijs{U}{p,\translated{v}{j}})$.
If we fix some $x$, we get \\ 
$\forall \, v{\leq}x \, \exists p \ (\axioms{V}{v} \rightarrow \bewijs{U}{p,\translated{v}{j}})$.
By \collection{\Sigma_1} we get the required\\
$\exists y\, \forall \, v{\leq} x \, \exists\, p{\leq}y\ 
(\axioms{V}{v}\rightarrow \bewijs{U}{\translated{v}{j}})$. It is not clear if 
$T\vdash \collection{\Sigma_1^-}$, parameter-free collection, is a necessary condition.\\

$\bullet$ $\sonetwo \not \vdash j: U\rhd_aV \rightarrow j:U\rhd_t V$.
A counter-example is given in \cite{Vis91}.\\

$\bullet$ $T\vdash j: U\rhd_t V \rightarrow j: U \rhd_{sa} V$, if $T\vdash \exp$.
If $V$ is reflexive, we get by Corollary \ref{coro:theimpliessmooth}
that 
$\vdash U \rhd_t V \leftrightarrow U\rhd_{sa} V$. However, different
interpretations are used to witness the different notions of interpretability
in this case. 
If $T\vdash \exp$, we reason as follows. We reason in $T$
and suppose that 
$\forall \varphi \, \forall p \, \exists p'\ 
(\bewijs{V}{p,\varphi} \rightarrow \bewijs{U}{p',\translated{\varphi}{j}})$.
We wish to see 
\begin{equation}\label{equa:shouldbesmooth}
\forall x \, \exists y\, \forall \, v{\leq} x \, \exists\, p{\leq}y\ 
(\axioms{V}{v}\rightarrow \bewijs{U}{\translated{v}{j}}) .
\end{equation}

So, we pick $x$ arbitrarily and consider\footnote{To see that
$\nu$ exists, we seem to also use some collection; we collect all the 
$v_i \leq x$ for which $\axioms{V}{v_i}$. However, it is not hard to see that 
we can consider $\nu$ also without collection since we use a natural coding.}
$\nu \eqbydef \bigwedge\limits_{\axioms{V}{v_i} \wedge v_i\leq x} v_i$.
Notice that in the worst case, for all $y\leq x$, we have \axioms{V}{y}, whence
the length of $\nu$ can be bounded by $x\cdot |x|$. Thus, $\nu$ itself 
can be bounded
by $x^{x}$, which exists whenever $T\vdash \exp$.
Clearly, $\exists p\ \bewijs{V}{p,\nu}$ whence by our assumption
$\exists p'\ \bewijs{U}{p',\translated{\nu}{j}}$.
In a uniform way, with just a slightly larger proof $p''$, every
$\translated{v_i}{j}$ can be extracted from the proof $p'$ of $\translated{\nu}{j}$.
We may take this $p''\approx y$ to obtain (\ref{equa:shouldbesmooth}).
Note that $T\vdash \exp$ is not a necessary condition since $\rhd_{t}$ implies $\rhd_{a}$ and if we have $B\Sigma_1$ the latter implies $\rhd_{sa}$.\\

$\bullet$ $\sonetwo \vdash j:U\rhd_{sa}V \rightarrow j:U\rhd_{st}V$. 
So, we wish to see that 
\[
\forall  x\, \exists y\, \forall \, \varphi {\leq} x\, \forall \, p{\leq}x \, \exists\, p'{\leq}y\ 
(\bewijs{V}{p,\varphi} \rightarrow \bewijs{U}{p',\translated{\varphi}{j}})
\]
from the assumption that $j: U \rhd_{sa}V$. So, we pick $x$ arbitrarily. If now for some 
$p\leq x$ we have \bewijs{V}{p,\varphi}, then clearly $\varphi \leq x$ and
all axioms $v_i$ of $V$ that occur in $p$ are $\leq x$. By our assumption $j:U\rhd_{sa}V$, we can
find a $y_0$ such that we can find proofs $p_i \leq y_0$ for all the \trans{v_i}{j}.
Now, with some sloppy notation, 
let $p^j [v_i^j / p_i]$ denote the $j$-translation of $p$ where each $j$-translated axiom $v_i^j$ is replaced by $p_i$. 

Clearly, $p^j [v_i^j / p_i]$ is a proof
for $\translated{\varphi}{j}$. The size of this proof can be estimated (again with sloppy 
notations):
\[
p^j [v_i^j / p_i] \leq 
p^j [v_i^j / y_0] \leq
(\translated{p}{j})^{|y_0|} \leq
(\translated{x}{j})^{|y_0|}.
\] 
The latter bound is clearly present in \sonetwo.
%
%
%
\end{proof}
%
%
%
%
%
We note that we have many admissible rules from one notion of
interpretability to another. For example,
by 
Buss's theorem on the provably total recursive functions
of \sonetwo, it is not hard to see that
\[
\sonetwo \vdash j:U \rhd_a V \Rightarrow \sonetwo\vdash j: U \rhd_t V .
\]

\noindent
In the rest of this paper, we shall at most places 
no longer write subscripts to the $\rhd$'s.
Our reading convention is then that we take that notion of interpretability that is
best to perform the argument. Often this is just smooth interpretability $\rhd_s$, 
which from now on is the notation for $\rhd_{sa}$. 

Moreover, in \cite{Vis91} some sort of conservation result concerning 
$\rhd_a$ and $\rhd_s$ is proved. For a considerable class of formulas 
$\varphi$ and theories $T$, 
and for a considerable class of arguments we have that
$T\vdash \varphi_a \Rightarrow T\vdash \varphi_s$. Here $\varphi_a$ denotes
the formula $\varphi$ using the notion $\rhd_a$ and likewise for $\varphi_s$.
Thus indeed, in many cases a sharp distinction between the notions involved is
not needed.

We could also consider the following notion of \inty.
\[
j : U\rhd_{st_1} V  \eqbydef 
\forall  x\, \exists y\, \forall \, \varphi {\leq} x\, \exists\, p'{\leq}y \ 
(\Box_V \varphi \rightarrow \bewijs{U}{p',\translated{\varphi}{j}})
\]
Clearly, $j : U\rhd_{st_1}V \rightarrow U\rhd_{st}V$. However, for the 
reverse implication one seems to need \collection{\Pi_1^-}. Also, a
straightforward proof of $U\vdash {\sf id}: U\rhd_{st_1}U$ seems 
to need \collection{\Pi_1^-}. Thus, the notion $\rhd_{st_1}$ seems to 
say more on the nature of a theory than on the nature of \inty.


\section{Cuts \index{cut} \index{definable cut} and induction}\label{sect:cuts}
Inductive reasoning is a central feature of everyday mathematical practice.
We are so used to it, that it enters a proof almost unnoticed.
It is when one works with weak theories and in the absence of sufficient induction,
that its all pervading nature is best felt.

A main tool to compensate for the lack of induction are the 
so-called definable cuts. They are definable initial segments 
of the natural numbers of a possibly non-standard model that possess some desirable properties that
we could not infer for all numbers to hold by means of induction.

The idea is really simple. So, if we can derive 
$\varphi (0) \wedge \forall x \ (\varphi (x) \rightarrow \varphi (x+1))$
and do not have access to an induction axiom for $\varphi$, we just
consider $J (x): \ \forall \, y{\leq}x\ \varphi (y)$. Clearly 
$J$ now defines an initial segment on which $\varphi$ holds.
As we shall see, for a lot of reasoning we can restrict ourselves
to initial segments rather than quantifying over all numbers.
%
%

\subsection{Basic properties of cuts}

Throughout the literature one can find some variations
on the definition of a cut. At some places, a cut is
only supposed to be an initial segment of the natural numbers.
At other places some additional closure properties are demanded.
By a well known technique due to Solovay (see for example
\cite{HP}) any definable initial segment can be shortened in 
a definable way, so that it has a lot of desirable closure properties.
Therefore, and as we almost always need the closure properties, we include them
in our definition.

\begin{definition}\label{definition:definableCut}
A definable $U$-cut is a formula $J(x)$ with 
only $x$ free, for which we have the following.
\begin{enumerate}
\item
$U\vdash J(0) \wedge \forall x\ (J (x) \rightarrow J (x+1))$

\item
$U\vdash J(x) \wedge y{\leq}x \rightarrow J(y)$

\item
$U \vdash J (x) \wedge J (y) \rightarrow J (x+ y) \wedge J (x \cdot y)$

\item
$U\vdash J(x) \rightarrow J(\omeja{x}{})$
\end{enumerate}
\end{definition}

We shall sometimes also write $x\in J$ instead of $J(x)$.
A first fundamental insight about cuts is the principle
of \emph{\Index{outside big, inside small}}. Although 
not every number $x$ is in $J$, we can find for every $x$ a proof 
$p_x$ that witnesses $x\in J$.

\begin{lemma}\label{lemm:outsidebiginsidesmall}
Let $T$ and $U$ be reasonable arithmetical theories and let $J$ be a 
$U$-cut. We have that 
\[
T\vdash \forall x\  \Box_U J(x) .
\]
Actually, we can have the quantifier over all cuts within the theory $T$, that 
is
\[
T\vdash \forall^{U{\mbox{-}}{\sf Cut}}J \, \forall x\ \Box_U J(x) .
\]
\end{lemma}

\begin{proof}
Let us start by making the 
quantifier \label{symb:forallcut}
$\forall^{U{\mbox{-}}{\sf Cut}} J$ 
a bit more precise.
By $\forall^{U{\mbox{-}}{\sf Cut}} J$ we shall mean
$\forall J\ (\Box_U \formal{Cut} (J) \rightarrow \ldots )$. Here $\formal{Cut}(J)$ is the 
definable function that sends the code of a formula $\chi$ with 
one free variable to the code of the formula that expresses that 
$\chi$ defines a cut.

For a number $a$, we start with the standard proof of $J(0)$. This proof
is combined with $a{-}1$ many instantiations of the standard proof of 
$\forall x\ (J (x) \rightarrow J(x+1))$. In the case of weaker theories, we
have to switch to efficient numerals to keep the bound of the proof within range.
\end{proof}

\begin{remark}\label{rema:outsidebig}
The proof sketch actually tells us that (provably in \sonetwo)
for every $U$-cut $J$, there is an $n\in \omega$ such that 
$\forall x\ \Box_{U,n} J(x)$.
\end{remark}

\begin{lemma}\label{lemm:cutsclosedunderterms}
Cuts are provably closed under terms, that is
\[
T \vdash \forall^{U{\mbox{-}}{\sf Cut}} J \, \forall^{\sf Term} t \ \Box_U
\forall \, \vec{x} {\in} J \ t(\vec{x}) \in J .
\]
\end{lemma}

\begin{proof}
By an easy induction on terms, fixing some $U$-cut $J$. Prima facie this looks
like a $\Sigma_1$-induction but it is easy to see that the proofs have 
poly-time (in $t$) bounds, whence the induction is $\Delta_0(\omega_1)$.
\end{proof}

\noindent
As all $U$-cuts are closed under \omeja{x}{} and the smash 
function $\sharp$, simply relativizing all 
quantors to a cut is an example of an interpretation of 
\sonetwo in $U$. We shall always 
denote both the cut and the interpretation that it defines by the same 
letter.

\subsection{Cuts and the \Index{Henkin construction}}

It is well known that we can perform the Henkin construction 
in a rather weak meta-theory. As the Henkin model has 
a uniform description, we can link it to interpretations.
The following theorem makes this precise.

\begin{theorem}\label{theo:informalhenkin}
If $U \vdash \cons{V}$, then $U\rhd V$.
\end{theorem}

Early treatments of this theorem were given in \cite{wang51} and
\cite{HB}. A first fully formalized version was given in 
\cite{Fef60}.
A proof of Theorem \ref{theo:informalhenkin}
would closely follow the Henkin construction.

Thus, first
the language of $V$ is extended so that it contains a witness 
$c_{\exists x \varphi(x)}$ for
every existential sentence $\exists x \ \varphi(x)$. 
Then we can extend $V$ to a maximal
consistent $V'$ in the enriched language, containing all sentences
of the form 
$\exists x \varphi (x) \rightarrow \varphi (c_{\exists x \varphi(x)})$.
This $V'$ can be seen as a term model with a corresponding 
truth predicate. Clearly, if $V\vdash \varphi$ then
$\varphi \in V'$. It is not hard to see that $V'$ is representable 
(close inspection yields a $\Delta_2$-representation) in $U$.

At first sight the argument uses quite some induction in extending 
$V$ to $V'$. Miraculously enough, the whole argument can be adapted
to \sonetwo. The trick consists in replacing the use of induction
by employing definable cuts as is explained above.
We get the following theorem.

\begin{theorem}\label{theo:formalhenkin}
\index{Henkin construction!in weak theories}
For any numberizable theories $U$  and $V$, we have that 
\[
\sonetwo\vdash \Box_U \cons{V} \rightarrow \exists k\ 
(k:U\rhd V  \ \& \ \forall \varphi \  
\Box_U ( \Box_V \varphi \rightarrow \translated{\varphi}{k})).
\]
\end{theorem}

\begin{proof}
A proof can be found in \cite{Vis91}. Actually something stronger is proved there. Namely, that
for some standard number $m$ we have
\[
\forall \varphi \, \exists \, p {\leq} \omeja{\varphi}{m} \ \bewijs{U}{p, \Box_V \varphi \rightarrow 
\translated{\varphi}{k}} .
\]
\end{proof}

As cuts have nice closure properties, many arguments can be
performed within that cut. The numbers in the cut will, so to say, play
the role of the normal numbers. It turns out that
the whole Henkin argument can be
carried out using only the consistency on a 
\index{Henkin construction!on a cut}cut.

\label{symb:cutcons}
We shall write $\Box^J_T \varphi$ for 
$\exists \, p{\in} J \ \bewijs{T}{p,\varphi}$. Thus, it is
also clear what ${\Diamond}^J_T \varphi$  and 
\cutcons{V}{J} mean.

\begin{theorem}\label{theo:kuthenkin}
We have Theorem \ref{theo:formalhenkin} also in the 
following form.
\[
T\vdash \forall^{U{\mbox{-}}{\sf Cut}}I\ \Big[ \Box_U \cutcons{V}{I} \rightarrow \exists k\ 
(k:U\rhd V  \ \& \ \forall \varphi \  
\Box_U ( \Box_V \varphi \rightarrow \translated{\varphi}{k}))\Big]
\]
\end{theorem}

\begin{proof}
By close inspection of the proof of Theorem \ref{theo:formalhenkin}.
All operations on hypothetical proofs $p$ can be bounded by 
some \omeja{p}{k}, for some standard $k$. As $I$ is closed under
\omeja{x}{}, all the bounds remain within $I$.
\end{proof}
We conclude this subsection with two asides, closely related to the
Henkin construction.

\begin{lemma}\label{lemm:conpredicate}
Let $U$ contain \buss. We have
that 
$U \vdash \cons{\pred}$.
Here, $\cons{\pred}$ is a natural formalization of the statement that predicate logic 
is consistent.
\end{lemma}

\begin{proof}
By defining a simple (one-point) model within \sonetwo.
\end{proof}

\begin{remark}\label{rema:equalitypreserving}
If $U$ proves $L\Delta_2^0$, then 
it holds that $U\rhd V$
iff $V$ is interpretable in $U$ by some interpretation that maps
identity to identity.
\end{remark}

\begin{proof}
Suppose $j: U\rhd V$ with 
$j =\pair{\delta ,F}$. We can define
$j' \eqbydef \pair{\delta' , F'}$ with
$\delta' (x) \eqbydef \delta (x) \wedge \forall \, y{<}x\ 
(\delta (y) \rightarrow y \translated{\neq}{j} x)$.
$F'$ agrees with $F$ on all symbols except that it maps identity to identity.
By the minimal number principle we can prove
$\forall x \ (\delta (x) \rightarrow \exists x' \ 
(x' \translated{=}{j}x) \wedge \delta'(x))$, and thus
$\forall \vec{x} \ 
(\delta' (\vec{x}) \rightarrow (\translated{\varphi}{j} (\vec{x}) \leftrightarrow 
\translated{\varphi}{j'} (\vec{x})
))$ for all formulae $\varphi$.
\end{proof}

\section{\Index{Pudl\'ak's lemma}}

In this section we will state and prove what is known as Pudl\'ak's lemma. Moreover, we shall prove a very useful consequence of this lemma. Roughly speaking, Pudl\'ak's lemma tells us how interpretations bear on the models that they induce. Therefor, let us first see how interpretations and models are related.

\subsection{Interpretations and models}

We can view interpretations $j: U\rhd V$ as 
a way of defining uniformly a model $\cal N$ of $V$ inside a model
$\cal M$ of $U$. Interpretations in 
foundational papers mostly bear the guise of a uniform model construction.

\begin{definition}\label{symb:transmodel}
Let $j: U\rhd V$ with $j= \pair{\delta, F}$. If ${\cal M}\models U$, we
denote by \translated{{\cal M}}{j} the following model.
\begin{itemize}
\item 
$|\translated{{\cal M}}{j}|= \{ 
x\in |{\cal M}| \mid {\cal M} \models \delta (x) \}/\equiv$,
where $a\equiv b$\/ iff ${\cal M}\models a=^jb$.   
\item 
$\translated{{\cal M}}{j}\models R (\alpha_1, \ldots , \alpha_n)$ iff
${\cal M} \models F(R) (a_1, \ldots , a_n)$, for
some  $a_1\in\alpha_1$, {\ldots},  $a_n\in\alpha_n$.
\end{itemize}
\end{definition}

\noindent
The fact that $j : U\rhd V$ is now reflected in the 
observation that, whenever ${\cal M}\models U$, then
$ \translated{{\cal M}}{j}\models V$.

On many occasions viewing interpretations as uniform model constructions provides
the right heuristics.

\subsection{Pudl\'ak's isomorphic cut}

Pudl\'ak's lemma is central to many arguments in the field of interpretability logics.
It provides a means to compare a model $\cal M$ of $U$ and its internally 
defined model \translated{{\cal M}}{j} of $V$ if $j: U\rhd V$.
If $U$ has full induction, this comparison is fairly easy.

\begin{theorem}\label{theo:initialembedding}
Suppose $j:U\rhd V$ and $U$ has full induction. Let $\cal M$ be a model of $U$.
We have that ${\cal M} \initialsegment \translated{{\cal M}}{j}$ via a 
\index{end-extension}definable embedding.
\end{theorem}

\begin{proof}
If $U$ has full  induction and $j : U\rhd V$, we may by Remark \ref{rema:equalitypreserving}
actually assume that $j$ maps identity in $V$ to identity in $U$. Thus, we 
can define the following function.
\[
f \eqbydef \left\{ \begin{array}{l}0 \mapsto \translated{0}{j}\\
x+1 \mapsto f (x) \translated{+}{j} \translated{1}{j} 
\end{array}
\right.
\]

Now, by induction, $f$ can be proved to be total. Note that
full induction is needed here, as we have a-priori no bound on the
complexity of \translated{0}{j} and \translated{+}{j}. Moreover, it 
can be proved that $f (a + b) = f(a) \translated{+}{j} f(b)$,
$f (a \cdot b) = f (a) \, \translated{\cdot}{j} f (b)$ and 
that 
$y \translated{\leq}{j} f(b) \rightarrow \exists \, a{<}b\ f(a) =y$. 
In other words, that
$f$ is an isomorphism between its domain and its co-domain and the 
co-domain is an initial segment of \translated{{\cal M}}{j}.
\end{proof}

If $U$ does not have full induction, a comparison between $\cal M$ and 
\translated{{\cal M}}{j} is given by Pudl\'ak's lemma, first explicitly mentioned in
\cite{pudl85}. Roughly, Pudl\'ak's lemma says that in the general case, we can find a 
definable $U$-cut $I$ of $\cal M$ and a definable embedding 
$f: I \longrightarrow \translated{{\cal M}}{j}$ such that
$f[I] \initialsegment \translated{{\cal M}}{j}$.

In formulating the statement we have
to be careful as we can no longer assume that identity is mapped to
identity. A precise formulation of Pudl\'ak's lemma in terms of an
isomorphism between two initial segments can for example be found in
\cite{joo:prol00}. We have chosen here to formulate and prove
the most general syntactic consequence of Pudl\'ak's lemma, namely that 
$I$ and $f[I]$, as substructures of $\cal M$ and \translated{{\cal M}}{j} 
respectively,
make true the same $\Delta_0$-formulas.

In the proof of Pudl\'ak's lemma we shall make the quantifier 
$\exists^{j,J{\mbox{-}}{\sf function}} h$ explicit. It basically means that 
$h$ defines a function from a cut $J$ to 
the $\translated{=}{j}$-equivalence classes of
the numbers defined by the interpretation $j$. 

\begin{lemma}[Pudl\'ak's Lemma]\label{lemm:pudlak} 
\[
\sonetwo\vdash j: U\rhd V \rightarrow \exists^{U{\mbox{-}}{\sf Cut}} J \, 
\exists^{j,J{\mbox{-}}{\sf function}} h \, \forall^{\Delta_0} \varphi \
\Box_U \forall \, \vec{x}\in J\ (\translated{\varphi}{j} 
(h (\vec{x})) \leftrightarrow \varphi (\vec{x}))
\]
Moreover, the $h$ and $J$ can be obtained uniformly from $j$ by a function that 
is provably total in \sonetwo.
\end{lemma}

\begin{proof}
Again, by $\exists^{U{\mbox{-}}{\sf Cut}} J\ \psi$ we shall mean
$\exists J\ (\Box_U \formal{Cut} (J) \wedge \psi)$, where $\formal{Cut}(J)$ is the 
definable function that sends the code of a formula $\chi$ 
to the code of a formula that expresses that $\chi$ defines a cut.
We apply a similar strategy for quantifying over
$j,J$-functions. 
Given a translation $j$, the defining property for a relation $H$ to be a $j,J$-function 
is
\[
\forall \,  \vec{x}, y , y'{\in} J \ (H (\vec{x}, y) \ \& \ H (\vec{x} , y') \rightarrow 
y\translated{=}{j} y') .
\] 
We will often consider $H$ as a function $h$ and write for example 
$\psi (h (\vec{x}))$ instead of 
\[
\forall y\ ( H (\vec{x}, y) \rightarrow \psi (y)).
\]

The idea of the proof is very easy. Just map the numbers of $U$ via $h$ to the 
numbers of $V$ so that $0$ goes to $\translated{0}{j}$ and the mapping commutes with the 
successor relation. If we want to prove a property of this mapping, we might 
run into problems as the intuitive proof appeals to induction. And sufficient induction
is precisely what we lack in weaker theories.

The way out here is to just put all the 
properties that we need our function $h$ to possess into its definition. 
Of course, then the work is in checking that we still have a good definition.
Being good means here that the set of numbers on which $h$ is defined
induces a definable $U$-cut.

In a sense, we want an (definable) initial part of the numbers of $U$ to be
isomorphic under $h$ to an initial part of the numbers of $V$. Thus, $h$ should 
definitely commute with successor, addition and multiplication. Moreover, the image
of $h$ should define an initial segment, that is, be closed under the smaller than
relation. All these requirements are reflected in the definition of 
\formal{Goodsequence}. Let $\delta$ denote the domain specifier of the translatio $j$. We define
\[
\begin{array}{lll}
\formal{Goodsequence}(\sigma , x, y) 
& \eqbydef 
& \formal{lh}(\sigma) = x+1 \wedge \sigma_0 \translated{=}{j} \translated{0}{j}
\wedge \sigma_x \translated{=}{j} y \\
\ & \ &
\wedge \ \forall \, i{\leq}x \ \delta (\sigma_i) \\
\ & \ &
\wedge \ 
\forall \, i{<}x\ (\sigma_{i+1} \translated{=}{j} \sigma_i
\translated{+}{j} \translated{1}{j})\\
\ & \ &
\wedge \ \forall \, k{+}l {\leq} x \ 
(\sigma_k \translated{+}{j} \sigma_l \translated{=}{j} \sigma_{k+l})\\
\ & \ &
\wedge \ \forall \, k{\cdot}l {\leq} x \ 
(\sigma_k \translated{\cdot}{j} \sigma_l \translated{=}{j} \sigma_{k\cdot l})\\
\ &\ &
\wedge \ \forall a \ (a \translated{\leq}{j} y \rightarrow 
\exists \, i{\leq}x\ \sigma_i \translated{=}{j} a).
\end{array}
\]
Subsequently, we define
\[
\begin{array}{lll}
H(x,y) & \eqbydef &
\exists \sigma \ \formal{Goodsequence}(\sigma , x, y)\\ 
\ &\ &
\wedge \
\forall \sigma' \, \forall y' \ 
(\formal{Goodsequence} (\sigma' , x,y' ) \rightarrow 
y\translated{=}{j} y'),
\end{array}
\]
and
\[
J'(x) \eqbydef \forall \, x'{\leq}x \, \exists y \ H(x',y ).
\]
Finally, we define $J$ to be the closure of $J'$ under $+$, $\cdot$ and
$\omeja{x}{}$. Now that we have defined all the machinery we can start the real proof.
The reader is encouraged to see at what place which defining property is used in the proof.

We first note that $J'(x)$ indeed defines a $U$-cut. For 
$\Box_U J'(0)$ you basically need sequentiality of $U$, and the 
translations of the identity axioms and properties of $0$.

To see $\Box_U \forall x\ (J' (x) \rightarrow J' (x+1))$ is also not hard. It follows from the translation of basic properties 
provable in $V$, like $x=y \rightarrow x+1 = y+1$ and 
$x+ (y+1) = (x+y) +1$, etc. The other properties of Definition \ref{definition:definableCut} go similarly.

We should now see that $h$ is a $j,J$-function. This is quite easy, as we have
all the necessary conditions
present in our definition.
Actually, we have
\begin{equation}\label{equa:hbijection}
\Box_U \forall \, x,y{\in} J \ 
(h (x)\translated{=}{j} h(y) \leftrightarrow x=y)
\end{equation}
The $\leftarrow$ direction reflects that $h$ is a $j,J$-function.
The $\rightarrow$ direction follows from elementary reasoning in $U$
using the translation of 
basic arithmetical facts provable in $V$. So, if $x\neq y$, say 
$x < y$, then $x + (z+1) =y$ whence
$h (x) \translated{+}{j} h(z+1) \translated{=}{j} h(y)$ which implies
$h (x) \translated{\neq}{j} h(y)$.

We are now to see that for our $U$-cut $J$ and for our $j,J$-function $h$ 
we indeed have that\footnote{We use 
$h (\vec{x} )$ as short for $h(x_0), \cdots, h (x_n)$.} 
\[
\forall^{\Delta_0} \varphi \
\Box_U \forall \, \vec{x}{\in} J\ (\translated{\varphi}{j} 
(h (\vec{x})) \leftrightarrow \varphi (\vec{x})) .
\]

First we shall proof this using a seemingly $\Sigma_1$-induction. A closer 
inspection of the proof shall show that we can 
provide at all places sufficiently small bounds, so that actually an 
$\omeja{x}{}$-induction suffices.
We first proof the following claim.

\begin{claim}\label{clai:termpudlak}
$\forall^{\formal{Term}} t\ 
\Box_U \forall \, \vec{x}, y\in J\ (\translated{t}{j} 
(h (\vec{x})) 
\translated{=}{j} h (y) \leftrightarrow t (\vec{x}) =y)$
\end{claim}

\begin{proof}
The proof is by induction on $t$. The basis is trivial. 
To see for example 
\[
\Box_U \forall \, y{\in} J \ 
(\translated{0}{j} \translated{=}{j} h (y)
\leftrightarrow 0=y)
\]
we reason in $U$ as follows. By the definition of
$h$, we have that $h(0) \translated{=}{j} \translated{0}{j}$, and by 
(\ref{equa:hbijection}) we moreover see that 
$\translated{0}{j} \translated{=}{j} h (y)
\leftrightarrow 0=y$. The other base case, that is, when $t$ is an atom, is 
precisely (\ref{equa:hbijection}).

For the induction step, we shall only do $+$, as 
$\cdot$ goes almost completely the same. Thus, we assume that
$t (\vec{x}) = t_1 (\vec{x} ) + t_2 (\vec{x} )$ and set out to prove
\[
\Box_U \forall \, \vec{x}, y {\in}J\ 
(\translated{t_1}{j}(h(\vec{x})) \translated{+}{j} 
\translated{t_2}{j}(h(\vec{x}))
\translated{=}{j} h (y)
\leftrightarrow
t_1 (\vec{x} ) + t_2 (\vec{x})=y) .
\]

Within $U$:
\begin{itemize}
\item[$\leftarrow$]
If $t_1 (\vec{x} ) + t_2 (\vec{x})=y$, then by 
Lemma \ref{lemm:cutsclosedunderterms}, we can find $y_1$ and 
$y_2$ with 
$t_1(\vec{x})= y_1$ and $t_2(\vec{x})= y_2$. The induction hypothesis tells us
that 
$\translated{t_1}{j}(h (\vec{x})) \translated{=}{j}
h(y_1)$
and 
$\translated{t_2}{j}(h (\vec{x})) \translated{=}{j}
h(y_2)$.
Now by (\ref{equa:hbijection}), $h(y_1 + y_2)
\translated{=}{j}h(y)$ and by the definition of $h$ we get that
\[
\begin{array}{lll}
h(y_1 + y_2) &\translated{=}{j} & h(y_1) \translated{+}{j} h (y_2) \\
 & \translated{=}{j}_{\sf i.h.} & \translated{t_1}{j}(h (\vec{x}))\translated{+}{j} 
\translated{t_2}{j}(h (\vec{x}))\\
  & \translated{=}{j} &
\translated{(t_1(h(\vec{x})) + t_2 (h (\vec{x})))}{j}.
\end{array}
\]

\item[$\rightarrow$]
Suppose now
$\translated{t_1}{j}(h (\vec{x}))\translated{+}{j} 
\translated{t_2}{j}(h (\vec{x})) \translated{=}{j} h(y)$.
Then clearly $\translated{t_1}{j}(h (\vec{x})) \translated{\leq}{j} h(y)$
whence by the definition of $h$ we can find some $y_1\leq y$ such that
$\translated{t_1}{j}(h (\vec{x}))$\translated{=}{j} $h(y_1)$ and 
likewise for $t_2$ (using the translation of the 
commutativity of addition). The induction hypothesis now yields 
$t_1 (\vec{x}) = y_1$ and $t_2 (\vec{x}) = y_2$. By the definition 
of $h$, we get\\
$h (y) \translated{=}{j} h(y_1) \translated{+}{j} h (y_2) 
\translated{=}{j} h (y_1 + y_2)$, whence by (\ref{equa:hbijection}),
$y_1 + y_2 =y$, that is, 
$t_1(\vec{x}) + t_2 (\vec{x})=y$.
\end{itemize}
\end{proof}

We now prove by induction on $\varphi \in \Delta_0$ that 
\begin{equation}\label{equa:inducsie}
\Box_U \forall \, \vec{x}{\in} J\ (\translated{\varphi}{j} 
(h (\vec{x})) \leftrightarrow \varphi (\vec{x})).
\end{equation}

For the base case, we consider that 
$\varphi \equiv t_1(\vec{x}) + t_2 (\vec{x})$.
We can now use Lemma \ref{lemm:cutsclosedunderterms} to note that
\[
\Box_U \forall \, \vec{x} {\in }J\ 
(t_1 (\vec{x}) =t_2(\vec{x}) \leftrightarrow 
\exists \, y{\in}J \ (t_1(\vec{x}) = y \wedge t_2
(\vec{x}) = y))
\]
and then use Claim \ref{clai:termpudlak}, 
the 
transitivity of $=$ 
and its translation
to obtain the result.

The boolean connectives are really trivial, so we only need to consider 
bounded quantification. We show (still within $U$) that
\[
\forall \, y, \vec{z}{\in} J\ 
(\forall \, x{\translated{\leq}{j}} h(y)\ 
\translated{\varphi}{j} (x, h (\vec{z})) \leftrightarrow
\forall \, x{\leq} y \ \varphi (x, \vec{z})) .
\]

$\leftarrow$ Assume $\forall \, x {\leq} y \ \varphi (x, \vec{z})$ for 
some $y, \vec{z} \in J$. We are to show \\
$\forall \, x{\translated{\leq}{j}} h(y)\ 
\translated{\varphi}{j} (x, h (\vec{z}))$. Now, pick some 
$x \translated{\leq}{j} h(y)$ (the translation of the universal quantifier
actually gives us an additional $\delta (x)$ which we shall omit for the sake
of readability). Now by the definition of $h$ we find some $y'\leq y$
such that $h (y') =x$. As $y' \leq y$, by our assumption,
$\varphi (y', \vec{z})$ whence by the induction hypothesis
$\translated{\varphi}{j} (h (y'), h (\vec{z}))$, that is 
$\translated{\varphi}{j} ( x, h (\vec{z}))$. As $x$ was arbitrarily
$\translated{\leq}{j} h (y)$, we are done.

$\rightarrow$ Suppose 
$\forall \, x{\translated{\leq}{j}} h(y)\ 
\translated{\varphi}{j} (x, h (\vec{z}))$. We are to see that 
$\forall \, x{\leq} y \ \varphi (x, \vec{z}))$. So, pick
$x\leq y$ arbitrarily. Clearly $h(x)\translated{\leq}{j} h(y)$, whence,
by our assumption\\ 
$\translated{\varphi}{j} (h(x), h (\vec{z}))$ and by the induction 
hypothesis, $\varphi (x, \vec{z})$.
\medskip

Note that in our proof we have used twice a $\Sigma_1$-induction;
In Claim \ref{clai:termpudlak} and in proving (\ref{equa:inducsie}). Let us now see that we can dispense with the $\Sigma_1$ induction.

In both cases, at every induction step, a constant piece $p'$ of proof is added to
the total proof. This piece looks every time the same. Only some parameters in 
it have to be replaced by subterms of $t$. So, the addition to the total proof
can be estimated by $p'_a(t)$ which is about
\bigo{t^k} for some standard $k$ and indeed, our induction was really but a bounded one. Both our inductions went over syntax and whence are available in \sonetwo.

Note that in proving (\ref{equa:inducsie}) we
dealt with the bounded quantification by appealing to the induction
hypothesis only once, followed by a generalization. So, fortunately
we did not need to apply the induction hypothesis to all $x{\leq}y$, which
would have yielded an exponential blow-up.
\end{proof}

\begin{remark}\label{rema:theoremsenough}
Pudl\'ak's lemma is valid already if we employ the notion of 
theorems interpretability rather than smooth interpretability. 
If we work with theories in the language of arithmetic, we can do even better.
In this case, axioms interpretability can suffice. In order to get this, all 
arithmetical facts whose translations were used in the proof of
Lemma  \ref{lemm:pudlak} have to be promoted to the status of axiom.
However, a close inspection of the proof shows that these facts are very 
basic and that there are not so many of them.
\end{remark}

If $j$ is an interpretation with $j:\alpha \rhd \beta$, we shall 
sometimes call the corresponding 
isomorphic cut that is given by Lemma \ref{lemm:pudlak}, the
\index{Pudl\'ak cut}\emph{Pudl\'ak} cut of $j$ and denote it by the 
corresponding upper case letter $J$.


\subsection{A consequence of Pudl\'ak's Lemma}

The following consequence of Pudl\'ak's Lemma is simple, yet can be very useful. For simplicity we state the consequence for sentential extensions of some base theory $T$ extending \sonetwo. Thus, $\alpha \rhd \beta$ will be short for $(T+\alpha) \rhd (T+\beta)$.

\begin{lemma}\label{theorem:KeepWitnessSmall}
(In \sonetwo:) If $j: \alpha \rhd \beta$ then, for every $T+\beta$ cut $I$ there exists a $T+\alpha$ cut $J$ such that for every $\gamma$ we have that 
\[
j: \alpha \wedge \Box^J \gamma \rhd \beta \wedge \Box^I \gamma .
\] 
\end{lemma}

\begin{proof}
By a minor adaptation of the standard argument. First, we define ${\sf Goodsequence}$. 
\[
\begin{array}{lll}
{\sf Goodsequence}(\sigma , x, y) 
& := 
& {\sf lh}(\sigma) = x+1 \wedge \sigma_0 {=}^{j} {0}^{j}
\wedge \sigma_x {=}^{j} y \\
\ & \ &
\wedge \ 
\forall \, i{<}x\ (\sigma_{i+1} {=}^{j} \sigma_i
{+}^{j} {1}^{j})\\
\ & \ &
\wedge \ \forall \, k{+}l {\leq} x \ 
(\sigma_k {+}^{j} \sigma_l {=}^{j} \sigma_{k+l})\\
\ & \ &
\wedge \ \forall \, k{\cdot}l {\leq} x \ 
(\sigma_k {\cdot}^{j} \sigma_l {=}^{j} \sigma_{k\cdot l})\\
\ &\ &
\wedge \ \forall a \ (a {\leq}^{j} y \rightarrow 
\exists \, i{\leq}x\ \sigma_i {=}^{j} a)\\
\ & \  &
\wedge \ 
\forall \, i{<}x \ I^j(\sigma_i)
\end{array}
\]
Next, we define
\[
\begin{array}{lll}
H(x,y) & := &
\exists \sigma \ {\sf Goodsequence}(\sigma , x, y)\\ 
\ &\ &
\wedge \
\forall \sigma' \, \forall y' \ 
({\sf Goodsequence} (\sigma' , x,y' ) \rightarrow 
y {=}^{j} y'),
\end{array}
\]
and
\[
J'(x) := \forall \, x'{\leq}x \, \exists y \ H(x',y ).
\]
Finally, we define $J$ to be the closure of $J'$ under $+$, $\cdot$ and
$\omega_1(x)$. 

As before, one can see $H(x,y)$ as defining a function (modulo $=^j$), call it $h$, that defines an isomorphism between $J$ and the image of $J$. Moreover, in the definition of $\sf Goodsequence$ we demanded that the image of $h$ is a subset of $I$ in the clause $\forall \, i{<}x \ I^j(\sigma_i)$.

It is easy to see that $J'$ is closed under successor, that is, 
$J'(x)\to J'(x+1)$. We only comment on the new ingredient of the image of $h$ being a subset of $I$. However, $T+\beta \vdash I(x) \to I(x+1)$, as $I$ is a definable cut. As $j: T + \alpha \rhd T + \beta$, clearly $T+\alpha \vdash I^j(x) \to I^j(x+^j1^j)$ and indeed $J'$ is closed under successor.

\end{proof}

In the literature, Lemma \ref{theorem:KeepWitnessSmall} was known only for $I$ to be the trivial cut of all numbers defined by $x=x$.

\section{The \Index{Orey-H\'ajek characterizations}}\label{subs:oh1}

This final section contains the most substantial part of the paper. 
We consider the diagram from Figure \ref{pict:oh}.
It is well known that all the implications hold when both $U$ and
$V$ are reflexive. This fact is referred to as the Orey-H\'ajek
characterizations (\cite{Fef60}, \cite{orey61}, \cite{haje:inte71}, 
\cite{haje:inte72}) for \inty. However, for the $\Pi_1$-conservativity
part, we should also mention work by Guaspari, Lindstr\"om and Pudl\'ak
(\cite{Gua79}, \cite{lind79}, \cite{Lin84}, \cite{pudl85}).

In this section we shall comment on all the implications 
in Figure \ref{pict:oh}, and study 
the conditions on $U$, $V$ and the meta-theory, that are 
necessary or sufficient.

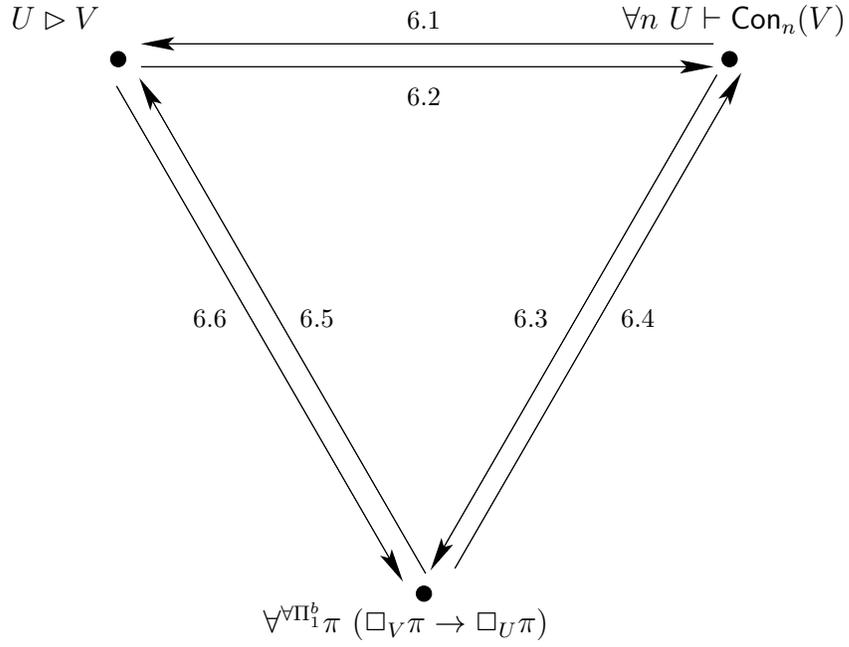
\begin{figure}
\begin{center}
\input{diagram.pstex_t}
\end{center}
\caption{Characterizations of interpretability. The labels at the arrows are references to where in the paper this arrow is proven and what the conditions are for the arrow to hold. Moreover, we will discuss which conditions should hold for the base theory so that the implications become formalizable.}\label{pict:oh}
\end{figure}

\begin{lemma}\label{lemm:OH1}
In \sonetwo we can prove
$\forall n \ \Box_{U} \restcons{V}{n} \rightarrow U \rhd V$.
\end{lemma}

\begin{proof}
The only requirement for this implication to hold, is that
$U \vdash \cons{\pred}$. But, 
by our assumptions on $U$ and by Lemma \ref{lemm:conpredicate}
this is automatically satisfied.

Let us first give the informal proof. Thus, let \axioms{V}{x} be the formula
that defines the axiom set of $V$.

We now apply a trick due to Feferman and consider the theory $V'$
that consists of those axioms of $V$ up to which we have evidence for 
their
consistency. 
Thus, $\axioms{V'}{x} \eqbydef \axioms{V}{x} \wedge \restcons{V}{x}$.

We shall now prove that 
$U\rhd V$ in two steps. First, we will see that 
\begin{equation}\label{equa:UprovesconV'}
U\vdash \cons{V'}.
\end{equation}
Thus, by Theorem \ref{theo:informalhenkin} we get that $U\rhd V'$. 
Second, we shall see that 
\begin{equation}\label{equa:visvaccent}
V=V'.
\end{equation}

To see (\ref{equa:UprovesconV'}), 
we reason in $U$, and assume for a contradiction that \bewijs{V'}{p,\bot}
for some proof $p$. We consider the largest axiom $v$ that occurs in $p$.
By assumption we have (in $U$) that \axioms{V'}{v} whence 
\restcons{V}{v}.
But, as clearly $V' \subseteq V$, we 
see that
$p$ is also a $V$-proof.
We can now obtain a cut-free proof $p'$ of $\bot$. Clearly
\bewijs{\restrict{V}{v}}{p',\bot} and we
have our contradiction. 

If $V'$ is empty, we cannot consider $v$. But in this case,
$\cons{V'}\leftrightarrow \cons{\pred}$, and by assumption,
$U\vdash \cons{\pred}$.

We shall now see (\ref{equa:visvaccent}). 
Clearly $\nat \models \axioms{V'}{v} \rightarrow \axioms{V}{v}$ for any
$v\in \nat$. To see that the converse also holds, we reason as follows.

Suppose $\nat \models \axioms{V}{v}$.
By assumption $U\vdash \restcons{V}{v}$, whence
\restcons{V}{v} holds on any model $\cal M$ of $U$.
We now \index{end-extension}observe 
that \nat is an \Index{initial segment} of (the numbers of) any 
model $\cal M$ of $U$, that is,
\begin{equation}\label{equa:initieel}
\nat \initialsegment {\cal M} .
\end{equation}
As ${\cal M}\models \restcons{V}{v}$ and as \restcons{V}{v} is a 
$\Pi_1$-sentence, we see that also 
$\nat \models \restcons{V}{v}$. By assumption we had 
$\nat \models \axioms{V}{v}$, thus we get that 
$\nat \models \axioms{V'}{v}$. We conclude that
\begin{equation}\label{equa:V=V'}
\nat \models \axioms{V}{x} \leftrightarrow \axioms{V'}{x}
\end{equation}
whence, that $V=V'$. As $U \vdash \cons{V'}$, we get by Theorem 
\ref{theo:informalhenkin} that $U\rhd V'$. We may thus infer the 
required $U\rhd V$.

\medskip

It is not possible to directly formalize the informal proof.
At (\ref{equa:V=V'}) we concluded that $V=V'$. This actually uses some form 
of $\Pi_1$-reflection which is manifested in (\ref{equa:initieel}). The lack
of reflection in the formal environment will be compensated by another sort of 
reflection, as formulated in Theorem \ref{theo:formalhenkin}.

Moreover, to see (\ref{equa:UprovesconV'}), we had to use a cut 
elimination. To avoid this, we shall need a sharper version of 
Feferman's trick. 

Let us now start with the formal proof sketch and refer to \cite{Vis91} for more details. 
We shall reason in $U$. Without any induction
we conclude 
$\forall x\ ({\sf Con}_x (V) \rightarrow {\sf Con}_{x+1}(V))$
or
$\exists x\ ({\sf Con}_x(V)\wedge \Box_{V,x+1}\bot)$. In both cases we
shall sketch a \Index{Henkin construction}.
\medskip

If 
$\forall x\ ({\sf Con}_x (V) \rightarrow {\sf Con}_{x+1}(V))$
and also ${\sf Con }_0(V)$, we can find a cut $J(x)$ with
$J(x)\rightarrow {\sf Con}_x(V)$.
We now consider the following non-standard proof predicate.
\[
\Box^\ast_W \varphi := \exists\, x{\in}J\ \Box_{W,x}\varphi
\]
We note that we have ${\sf Con}^\ast(V)$, where 
${\sf Con}^\ast (V)$ of course denotes 
$\neg (\exists \, x{\in}J \ \Box_{V,x}\bot)$.
As always, we extend the language on $J$ by adding witnesses and define
a series of theories in the usual way. That is, by adding more and more sentences 
(in $J$) to our theories while staying consistent (in our non-standard sense).  
\begin{equation}\label{equa:goinghenkin}
V=V_0 \subseteq V_1 \subseteq V_2 \subseteq \cdots 
\mbox{with ${\sf Con}^\ast (V_i)$}
\end{equation}
We note that
$\Box^\ast_{V_i}\varphi$ and 
$\Box^\ast_{V_i}\neg \varphi$ is not possible, and that for
$\varphi \in J$ we can not have ${\sf Con}^*(\varphi \wedge
\neg \varphi)$. These observations seem to be too trivial to make, but
actually many a non-standard proof predicate encountered in the
literature does prove the consistency of inconsistent theories.

As always, the sequence (\ref{equa:goinghenkin})
defines a cut $I\subseteq J$, that induces a Henkin set $W$
and we can relate our required interpretation $k$ to this 
Henkin set as was, for example, done in \cite{Vis91}.
\medskip

We now consider the case that for some fixed $b$ we have
${\sf Con }_b(V)\wedge \Box_{V,b+1}\bot$. We note that we can see the
uniqueness of this $b$ without using any substantial induction.
Basically, we shall now do the same construction as before only that
we now possibly stop at $b$. 

For example the cut $J(x)$ will now be replaced by $x\leq b$.
Thus, we may end up with a truncated 
Henkin set $W$. But this set is complete with respect to relatively 
small formulas.
Moreover, $W$ is certainly closed under subformulas and substitution 
of witnesses. Thus, $W$ is sufficiently large to define the 
required interpretation $k$.
\medskip

In both cases we can perform the
following reasoning.
\[
\begin{array}{llll}
\Box_V \varphi  & \rightarrow & \exists x\ \Box_{V,x} \varphi\\   
  & \rightarrow & \exists x\ \Box_U ({\sf Con}_x(V) \wedge \Box_{V,x} \varphi)\\   
  & \rightarrow &  \Box_U \Box^\ast_V \varphi\\   
  & \rightarrow &  \Box_U \varphi^k & \mbox{by Theorem \ref{theo:formalhenkin}.}\\   
\end{array}
\]
The remarks from \cite{Vis91} on the bounds of our proofs are 
still applicable and we thus obtain a smooth interpretation.
\end{proof}

\begin{lemma}\label{lemm:OH2}
In the presence of $\expo$, we can prove that for reflexive $U$,
$U\rhd V \rightarrow \forall x\ \Box_U \restcons{V}{x}$.
\end{lemma}

\begin{proof}
The
informal argument is conceptually very clear and we have depicted it in Figure
\ref{pict:prooftransformations}.
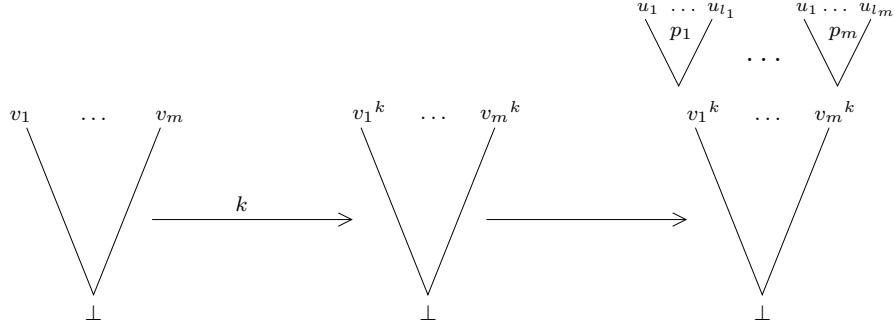
\begin{figure}
\begin{center}
\input{prooftransformations.pstex_t}
\end{center}
\caption{Transformations on proofs}\label{pict:prooftransformations}
\end{figure}
The accompanying reasoning is as follows.

We assume $U \rhd V$, whence 
for some $k$ we have
$k : U \rhd V$. Thus, for axioms interpretability we find that
$\forall u \, \exists p \  (\axioms{V}{u} \rightarrow \bewijs{U}{p, 
\translated{u}{k}})$.
We are now to see that $\forall x \ U \vdash  
\restcons{V}{x}$. So, we fix some $x$.
By our assumption we get that for some $l$, that
\begin{equation}\label{equa:firstbesigma}
\forall \, u{\leq} x \, \exists p \  (\axioms{V}{u} \rightarrow 
\bewijs{U,l}{p, \translated{u}{k}}).
\end{equation}
This formula is actually equivalent to the $\Sigma_1$-formula 
\begin{equation}\label{equa:secondbesigma}
\exists n \, \forall \, u{\leq} x \, \exists \, p{\leq} n  \  (\axioms{V}{u} \rightarrow 
\bewijs{U,l}{p, \translated{u}{k}})
\end{equation}
from which 
we may conclude by provable $\Sigma_1$-completeness,
\begin{equation}\label{equa:firstsigmavoll}
U\vdash \exists n \, \forall \, u{\leq} x \, \exists \, p{\leq} n \  (\axioms{V}{u} 
\rightarrow \bewijs{U,l}{p, \translated{u}{k}}).
\end{equation}

We now reason in $U$ and
suppose that
there is some $V,x$-proof $p$ of $\bot$. The assumptions in $p$ are axioms
$v_1\ldots v_m$ of $V$, with each $v_i \leq x$. 
Moreover, all the formulas $\psi$ in $p$ have $\rho(\psi)\leq x$.
By Lemma \ref{lemm:translatedproofs}, this $p$ transforms
to a proof \translated{p}{k} of \translated{\bot}{k} which is again $\bot$.

The assumptions in \translated{p}{k} are now among the 
$\translated{v_1}{k}\ldots \translated{v_m}{k}$. 
By Remark \ref{rema:translation} we get that for some $n'$
depending on $x$ and $k$, we have that all the axioms in
$p^k$ are $\leq n'$ and all the $\psi$ occurring in $p^k$
have $\rho(\psi)\leq n'$.

Now by (\ref{equa:firstsigmavoll}), we  have $U,l$-proofs $p_i \leq n$ 
of $\translated{v_i}{k}$. The 
assumptions in the $p_i$ are axioms of $U$. Clearly all of these 
axioms are $\leq l$. We can now form a 
$U,l{+}n'$-proof $p'$ of $\bot$ 
by substituting all the $p_i$ for the $(v_i)^k$. Thus we have shown
\bewijs{\restrict{U}{l+n'}}{p',\bot}. But this clearly contradicts the 
reflexivity of $U$.

The informal argument is readily formalized to obtain
$T \vdash U\rhd V \rightarrow \forall x\ \Box_U \cons{\restrict{V}{x}}$. 
However there are 
some  subtleties.

First of all, to conclude that (\ref{equa:firstbesigma})
is equivalent to (\ref{equa:secondbesigma}), a genuine application 
of \collection{\Sigma_1} is needed. If $U$ lacks 
\collection{\Sigma_1}, we have to switch to smooth interpretability to still
have the implication valid. Smoothness then automatically also provides the
$l$ that we used in \ref{equa:firstbesigma}.

In addition we need that $T$ proves the 
totality of exponentiation. For weaker theories, we only have
provable $\exists \Sigma_1^b$-completeness. But if $\axioms{V}{u}$ is
$\Delta_1^b$, we can only guarantee that
$\forall \, u{\leq} m \, \exists \, p{\leq} n \  (\axioms{V}{u} 
\rightarrow \bewijs{U}{p, \translated{u}{k}})$ is $\Pi_2^b$. As far as
we know, exponentiation is needed to prove 
$\exists \Pi_2^b$-completeness.

All other transformations of objects in our proof only require the 
totality of \omeja{x}{}.
%
\end{proof}

The assumption that $U$ is reflexive can in a sense not be dispensed with.
That is, if
\begin{equation}
\forall V\ (U\rhd V \rightarrow \forall x\  \Box_U \restcons{V}{x}),
\end{equation}
then $U$ is reflexive, as clearly $U\rhd U$. In a similar way we see that if
\begin{equation}\label{equa:invokingshavru}
\forall U\ (U\rhd V \rightarrow \forall x\  \Box_U \restcons{V}{x}),
\end{equation}
then $V$ is reflexive. However, $V$ being reflexive could never be a sufficient
condition for (\ref{equa:invokingshavru}) to hold, as we know from
\cite{Sha97} that interpreting reflexive theories in finitely many axioms is
complete $\Sigma_3$.


\begin{lemma}\label{lemm:OH3}
In \sonetwo we can prove
$\forall x\ \Box_U \restcons{V}{x}\rightarrow
\forall^{\forall \Pi_1^b}\pi \ (\Box_V \pi \rightarrow \Box_U \pi)$.
\end{lemma}

\begin{proof}
There are no conditions on $U$ and $V$ for this implication to hold.
We shall directly give the formal proof as the informal proof does not 
give a clearer picture.

Thus, we reason 
in \sonetwo and assume
$\forall x \ \Box_U \restcons{V}{x}$. Now we consider 
any $\pi \in \forall \Pi_1^b$ 
such that $\Box_V \pi$. Thus, for some $x$ we have
$\Box_{\restrict{V}{x}}\pi$.
We choose $x$ large enough, so that we also have (see Remark \ref{rema:sigma})
\begin{equation}\label{equa:Uprovessigmacompleteness}
\Box_U (\neg \pi \rightarrow \Box_{\restrict{V}{x}}\neg \pi).
\end{equation}
As $\Box_{\restrict{V}{x}} \pi \rightarrow \Box_U \Box_{\restrict{V}{x}}\pi$, we 
also have that 
\begin{equation}\label{equa:UprovesVreflexive}
\Box_U \Box_{\restrict{V}{x}} \pi.
\end{equation}
Combining (\ref{equa:Uprovessigmacompleteness}), 
(\ref{equa:UprovesVreflexive})
and the assumption that 
$\forall x \ \Box_U \restcons{V}{x}$,
we see that indeed $\Box_U \pi$.
\end{proof}


\begin{lemma}\label{lemm:OH4}
In \sonetwo we can prove that for reflexive $V$ we have
\[
\forall^{\forall \Pi_1^b}\pi \ (\Box_V \pi \rightarrow \Box_U \pi)
\rightarrow \forall x\ \Box_U \restcons{V}{x}.
\]
\end{lemma}

\begin{proof}
If $V$ is 
reflexive and $\forall^{\forall \Pi_1^b}\pi \ (\Box_V \pi \rightarrow \Box_U \pi)$ then,
as for every $x$, \restcons{V}{\numeral{x}} 
is a $\forall \Pi_1^b$-formula,
also $\forall x\ \Box_U \restcons{V}{x}$.
\end{proof}

It is obvious that 
\begin{equation}\label{equa:Piconservativityimpliesreflexivity}
\forall U\ [\forall^{\forall \Pi_1^b}\pi \ (\Box_V \pi \rightarrow \Box_U \pi)
\rightarrow \forall x\ \Box_U \restcons{V}{x}]
\end{equation}
implies that $V$ is reflexive. Likewise,
\begin{equation}\label{equa:PiconservativityimpliesreflexivityV}
\forall V\ [\forall^{\forall \Pi_1^b}\pi \ (\Box_V \pi \rightarrow \Box_U \pi)
\rightarrow \forall x\ \Box_U \restcons{V}{x}]
\end{equation}
implies that $U$ is reflexive. However, $U$ being reflexive can never
be a sufficient condition for (\ref{equa:PiconservativityimpliesreflexivityV})
to hold. An easy counterexample is obtained by taking
$U$ to be \pra and $V$ to be \isig{1} as it is well-known that \isig{1} is provably $\Pi_2$ conservative over \pra and that \isig{1} is finitely axiomatized.

\begin{lemma}\label{lemm:OH5}
(In \sonetwo:) For reflexive $V$ we have
$\forall^{\forall \Pi_1^b} \pi\ 
(\Box_V \pi \rightarrow \Box_U \pi) \rightarrow U\rhd V$.
\end{lemma}

\begin{proof}
We know of
no direct proof of this implication. Also, all proofs in the literature 
go via Lemmata \ref{lemm:OH4} and \ref{lemm:OH1}, 
and hence use reflexivity of
$V$.
\end{proof}


In our context, the reflexivity of $V$ is not necessary, as 
$\forall U\ U \rhd \buss$ and \buss is not reflexive.

\begin{lemma}\label{lemm:OH6}
Let $U$ be a reflexive and sequential theory. We have in \sonetwo that
$U\rhd V  \rightarrow \forall^{\forall \Pi_1^b} \pi\ 
(\Box_V \pi \rightarrow \Box_U \pi)$.

If moreover $U \vdash \expo$ we also get
$U\rhd V \rightarrow \forall^{\Pi_1} \pi\ 
(\Box_V \pi \rightarrow \Box_U \pi)$.
If $U$ is not reflexive, we still have that 
$U\rhd V \rightarrow \exists^{U\mbox{-}{\sf Cut}} J \, 
\forall^{\Pi_1} \pi\ (\Box_V \pi \rightarrow \Box_U \pi^J)$.

For these implications, it is actually sufficient to work with the 
notion of theorems interpretability.
\end{lemma}

\begin{proof}
The intuition for the formal proof comes from Pudl\'ak's lemma, which in turn is tailored 
to compensate a lack of induction. We shall first give an informal proof sketch
if $U$ has full induction. Then we shall give the formal proof using Pudl\'ak's lemma.

If $U$ has full induction and $j : U\rhd V$,  we may assume by Remark
\ref{rema:equalitypreserving} assume that $j$ maps identity to identity.
Let ${\cal M}$ be an arbitrary model of $U$. By Theorem \ref{theo:initialembedding} we now see that 
${\cal M}\initialsegment \translated{{\cal M}}{j}$. If for some $\pi \in \Pi_1$,
$\Box_V \pi$ then by soundness $\translated{{\cal M}}{j}\models \pi$, whence 
${\cal M}\models \pi$. As ${\cal M}$ was an arbitrary model of $U$, we get by the completeness theorem that $\Box_U \pi$.


To transform this argument into a formal one, valid for weak theories, there
are two major adaptations to be made. 
First, the use of the soundness and completeness theorem
have to be avoided 
\negeer{(see also Remark \ref{rema:formalmodel})}. This can be done
by simply staying in the realm of provability. Secondly, we should get rid of
the use of full induction. This is done by switching to a cut in Pudl\'ak's lemma.
\medskip

Thus, the formal argument runs as follows. Reason in \sonetwo and assume $U\rhd V$.

We fix some $j: U \rhd V$. By Pudl\'ak's lemma, Lemma \ref{lemm:pudlak}, we
now find\footnote{Remark \ref{rema:theoremsenough} ensures us that
we can find them also in the case of theorems interpretability.} 
a definable $U$-cut $J$ and a $j,J$-function $h$ such that
\[
\forall^{\Delta_0} \varphi \
\Box_U \forall \, \vec{x}{\in} J\ (\translated{\varphi}{j} 
(h (\vec{x})) \leftrightarrow \varphi (\vec{x})) .
\]
We shall see that for this cut $J$ we have that
\begin{equation}\label{equa:blab}
\forall^{\Pi_1}\pi \ (\Box_V \pi \rightarrow \Box_U \pi^J).
\end{equation}
Therefore, we fix
some $\pi \in \Pi_1$ and assume $\Box_V \pi$. Let 
$\varphi (x) \in \Delta_0$ be such that 
$\pi = \forall x \ \varphi (x)$.
Thus we have $\Box_V \forall x \ \varphi (x)$, hence
by theorems interpretability
\begin{equation}\label{equa:bla}
\Box_U \forall x \ (\delta (x) \rightarrow \translated{\varphi}{j}(x)) .
\end{equation}
We are to see 
\begin{equation}\label{equa:blabla}
\Box_U \forall x\ (J(x) \rightarrow \varphi (x)).
\end{equation}
To see this, we reason in $U$ and fix $x$ such that $J(x)$. By definition
of $J$, $h (x)$ is defined. By the definition of $h$, we have 
$\delta (h(x))$, whence by (\ref{equa:bla}), 
$\translated{\varphi}{j}(h(x))$. Pudl\'ak's lemma now yields the desired 
$\varphi (x)$. As $x$ was arbitrary, we have proved (\ref{equa:blabla}).

So far, we have not used the reflexivity of $U$. We shall now see that
\[
\forall^{\forall \Pi_1^b} \pi\ (\Box_U \pi^J \rightarrow \Box_U \pi)
\]
holds for any $U$-cut $J$ whenever $U$ is reflexive. For this purpose, we 
fix some $\pi \in \forall\Pi_1^b$, some $U$-cut $J$ and assume
$\Box_U \pi^J$. Thus, 
$\exists n \ 
\Box_{\restrict{U}{n}} \pi^J$ and also
$\exists n \ \Box_U
\Box_{\restrict{U}{n}} \pi^J$.
If $\pi = \forall x \ \varphi (x)$ with $\varphi (x) \in \Pi_1^b$, we
get 
$\exists n \ \Box_U
\Box_{\restrict{U}{n}} \forall x \ (x\in J \rightarrow \varphi (x))$, 
whence also
\[
\exists n \ \Box_U  \forall x \ 
\Box_{\restrict{U}{n}}(x\in J \rightarrow \varphi (x)) .
\]
By Lemma \ref{lemm:outsidebiginsidesmall} and Remark 
\ref{rema:outsidebig}, for large enough $n$, this implies
\[
\exists n \ \Box_U  \forall x \ 
\Box_{\restrict{U}{n}}  \varphi (x)
\]
and by Lemma \ref{lemm:weakpireflexion} (only here we use that
$\pi \in \forall \Pi_1^b$) we obtain the required
$\Box_U \forall x\ \varphi (x)$.
\end{proof}



Again, by \cite{Sha97} we note that $V$ being reflexive can never be
a sufficient condition for 
$\forall U\ [U\rhd V  \rightarrow \forall^{\forall \Pi_1^b} \pi\ 
(\Box_V \pi \rightarrow \Box_U \pi)]$.

\medskip

The main work on the Orey-H\'ajek characterization has now been done.
We can easily extract some useful, mostly well-known corollaries.

\begin{corollary}\label{coro:fijnoh}
If $U$ is a reflexive theory, then 
\[
T \vdash U\rhd V \leftrightarrow \forall x\ \Box_U \restcons{V}{x}.
\]
Here $T$ contains $\expo$ and $\rhd$ denotes smooth interpretability.
\end{corollary}

\begin{corollary}\label{coro:Vrefl}
(In \sonetwo:) If $V$ is a reflexive theory, then the following are equivalent.
\begin{enumerate}
\item \label{item:int}
$U\rhd V$

\item \label{item:picons}
$\exists^{U\mbox{-}{\sf Cut}} J \, 
\forall^{\Pi_1} \pi\ (\Box_V \pi \rightarrow \Box_U \pi^J)$

\item \label{item:oh}
$\exists^{U\mbox{-}{\sf Cut}} J \, \forall x\ \Box_U {\sf Con}_x^J(V)$

\end{enumerate}
\end{corollary}

\begin{proof}
This is part of Theorem 2.3 from \cite{Sha97}. 
$(\ref{item:int}) \Rightarrow (\ref{item:picons})$ is already proved in
Lemma \ref{lemm:OH6}, 
$(\ref{item:picons}) \Rightarrow (\ref{item:oh})$ follows
from the transitivity of $V$ and 
$(\ref{item:oh}) \Rightarrow (\ref{item:int})$ is a sharpening of 
Lemma \ref{lemm:OH1}.
which closely follows Theorem \ref{theo:kuthenkin}.
Note that $\rhd$ may denote denote smooth or theorems interpretability.
\end{proof}

\begin{corollary}\label{coro:theimpliessmooth}
If $V$ is reflexive, then 
\[
\sonetwo \vdash U \rhd_t V \leftrightarrow U\rhd_s V .
\]
\end{corollary}

\begin{proof}
By Remark \ref{rema:theoremsenough} and Corollary \ref{coro:Vrefl}.
\end{proof}

\begin{corollary}\label{coro:UVrefl}
If $U$ and $V$ are both reflexive theories we have that the 
following are provably equivalent in \sonetwo.
\begin{enumerate}
\item \label{item:int1}
$U\rhd V$

\item \label{item:picons1}
$\forall^{\forall \Pi_1^b} \pi\ (\Box_V \pi \rightarrow \Box_U \pi)$

\item \label{item:oh1}
$\forall x\ \Box_U \restcons{V}{x}$

\end{enumerate}
\end{corollary}

\begin{proof}
If we go $(\ref{item:int1}) \Rightarrow (\ref{item:picons1}) \Rightarrow 
(\ref{item:oh1}) \Rightarrow (\ref{item:int1})$ we do not need the
totality of $\expo$ that was needed for 
$(\ref{item:int1}) \Rightarrow (\ref{item:oh1})$.
\end{proof}

As an application we can, for example,
see that $\pa \rhd \pa + \formal{InCon}(\pa)$. It is well known
that \pa is essentially reflexive which means that any finite extension of it is reflexive. So, we use Corollary 
\ref{coro:UVrefl} and, it is sufficient to show that
$\pa + \formal{InCon}(\pa)$ is $\Pi_1$-conservative over
\pa.

So, suppose that 
$\pa + \formal{InCon}(\pa) \vdash \pi$ for some
$\Pi_1$-sentence $\pi$. In other words
$\pa \vdash \Box \bot \rightarrow \pi$. We shall now see
that 
$\pa \vdash \Box \pi \rightarrow \pi$, which by L\"ob's Theorem gives us
$\pa \vdash \pi$.

Thus, in \pa, assume $\Box \pi$. Suppose for a contradiction
that $\neg \pi$. By $\Sigma_1$-completeness we also get
$\Box \neg \pi$, which yields $\Box \bot$ with the assumption
$\Box \pi$. But we have $\Box \bot \rightarrow \pi$ and we conclude $\pi$.
A contradiction, so that indeed $\pa \rhd \pa + \formal{InCon}(\pa)$.

\section*{Acknowledgements}
I am grateful to Lev Beklemishev, F\'elix Lara and Albert Visser for pointers to the literature and helpful discussions. 

This research has been funded by Grant 2014 SGR 437 from the Catalan government and by Grant MTM2014-59178-P from the Spanish government.

\bibliographystyle{plain}
\bibliography{Oldlevref}

\end{document}

%% file: transitive.pstex_t
\begin{picture}(0,0)%
\includegraphics{transitive.pstex}%
\end{picture}%
\setlength{\unitlength}{2763sp}%
\begingroup\makeatletter\ifx\SetFigFont\undefined%
\gdef\SetFigFont#1#2#3#4#5{%
  \reset@font\fontsize{#1}{#2pt}%
  \fontfamily{#3}\fontseries{#4}\fontshape{#5}%
  \selectfont}%
\fi\endgroup%
\begin{picture}(6762,3753)(1501,-4633)
\put(7126,-1336){\makebox(0,0)[lb]{\smash{\SetFigFont{12}{14.4}{\rmdefault}{\mddefault}{\updefault}{\color[rgb]{0,0,0}$\ldots$}%
}}}
\put(7876,-1036){\makebox(0,0)[lb]{\smash{\SetFigFont{8}{9.6}{\rmdefault}{\mddefault}{\updefault}{\color[rgb]{0,0,0}$p_m$}%
}}}
\put(6451,-1036){\makebox(0,0)[lb]{\smash{\SetFigFont{8}{9.6}{\rmdefault}{\mddefault}{\updefault}{\color[rgb]{0,0,0}$p_1$}%
}}}
\put(3601,-1861){\makebox(0,0)[lb]{\smash{\SetFigFont{8}{9.6}{\rmdefault}{\mddefault}{\updefault}{\color[rgb]{0,0,0}$\translated{t_1}{k}\ \ \ \ \ldots \ \ \ \ \translated{t_m}{k}$}%
}}}
\put(4201,-3661){\makebox(0,0)[lb]{\smash{\SetFigFont{8}{9.6}{\rmdefault}{\mddefault}{\updefault}{\color[rgb]{0,0,0}$\translated{u}{k}$}%
}}}
\put(6601,-1861){\makebox(0,0)[lb]{\smash{\SetFigFont{8}{9.6}{\rmdefault}{\mddefault}{\updefault}{\color[rgb]{0,0,0}$\translated{t_1}{k}\ \ \ \ \ \ldots \ \ \ \ \ \translated{t_m}{k}$}%
}}}
\put(2551,-2611){\makebox(0,0)[lb]{\smash{\SetFigFont{8}{9.6}{\rmdefault}{\mddefault}{\updefault}{\color[rgb]{0,0,0}$k$}%
}}}
\put(5701,-2611){\makebox(0,0)[lb]{\smash{\SetFigFont{8}{9.6}{\rmdefault}{\mddefault}{\updefault}{\color[rgb]{0,0,0}$j$}%
}}}
\put(4276,-4561){\makebox(0,0)[b]{\smash{\SetFigFont{11}{13.2}{\rmdefault}{\mddefault}{\updefault}{\color[rgb]{0,0,0}a $T$-proof}%
}}}
\put(7351,-4561){\makebox(0,0)[b]{\smash{\SetFigFont{11}{13.2}{\rmdefault}{\mddefault}{\updefault}{\color[rgb]{0,0,0}an $S$-proof}%
}}}
\put(7126,-3661){\makebox(0,0)[lb]{\smash{\SetFigFont{8}{9.6}{\rmdefault}{\mddefault}{\updefault}{\color[rgb]{0,0,0}$\translated{(\translated{u}{k})}{j}$}%
}}}
\put(1651,-4486){\makebox(0,0)[b]{\smash{\SetFigFont{11}{13.2}{\rmdefault}{\mddefault}{\updefault}{\color[rgb]{0,0,0}$\axioms{U}{u}$}%
}}}
\put(1501,-2821){\makebox(0,0)[lb]{\smash{\SetFigFont{8}{9.6}{\rmdefault}{\mddefault}{\updefault}{\color[rgb]{0,0,0}$u$}%
}}}
\end{picture}

%% file: smoothnew.pstex_t
\begin{picture}(0,0)%
\includegraphics{smoothnew.pstex}%
\end{picture}%
\setlength{\unitlength}{2368sp}%
\begingroup\makeatletter\ifx\SetFigFont\undefined%
\gdef\SetFigFont#1#2#3#4#5{%
  \reset@font\fontsize{#1}{#2pt}%
  \fontfamily{#3}\fontseries{#4}\fontshape{#5}%
  \selectfont}%
\fi\endgroup%
\begin{picture}(5550,4017)(1126,-3283)
\put(3601,-1711){\makebox(0,0)[lb]{\smash{\SetFigFont{10}{12.0}{\rmdefault}{\mddefault}{\updefault}{\color[rgb]{0,0,0}$j:U\rhd_{sa} V$}%
}}}
\put(1201,-211){\makebox(0,0)[lb]{\smash{\SetFigFont{10}{12.0}{\rmdefault}{\mddefault}{\updefault}{\color[rgb]{0,0,0}$j:U\rhd_t V$}%
}}}
\put(1201,-3211){\makebox(0,0)[lb]{\smash{\SetFigFont{10}{12.0}{\rmdefault}{\mddefault}{\updefault}{\color[rgb]{0,0,0}$j:U\rhd_a V$}%
}}}
\put(6676,-1711){\makebox(0,0)[lb]{\smash{\SetFigFont{10}{12.0}{\rmdefault}{\mddefault}{\updefault}{\color[rgb]{0,0,0}$j:U\rhd_{st} V$}%
}}}
\put(3301,-736){\makebox(0,0)[lb]{\smash{\SetFigFont{10}{12.0}{\rmdefault}{\mddefault}{\updefault}{\color[rgb]{0,0,0}$\exp$}%
}}}
\put(3301,-2611){\makebox(0,0)[lb]{\smash{\SetFigFont{10}{12.0}{\rmdefault}{\mddefault}{\updefault}{\color[rgb]{0,0,0}$\collection{\Sigma_1}$}%
}}}
\put(1126,539){\makebox(0,0)[lb]{\smash{\SetFigFont{12}{14.4}{\rmdefault}{\mddefault}{\updefault}{\color[rgb]{0,0,0}In \sonetwo :}%
}}}
\end{picture}

%% file: diagram.pstex_t
\begin{picture}(0,0)%
\includegraphics{diagram.pstex}%
\end{picture}%
\setlength{\unitlength}{2526sp}%
\begingroup\makeatletter\ifx\SetFigFont\undefined%
\gdef\SetFigFont#1#2#3#4#5{%
  \reset@font\fontsize{#1}{#2pt}%
  \fontfamily{#3}\fontseries{#4}\fontshape{#5}%
  \selectfont}%
\fi\endgroup%
\begin{picture}(7179,6222)(1351,-7267)
\put(1351,-1261){\makebox(0,0)[lb]{\smash{\SetFigFont{12}{14.4}{\rmdefault}{\mddefault}{\updefault}{\color[rgb]{0,0,0}$U\rhd V$}%
}}}
\put(3826,-7186){\makebox(0,0)[lb]{\smash{\SetFigFont{12}{14.4}{\rmdefault}{\mddefault}{\updefault}{\color[rgb]{0,0,0}$\forall^{\forall \Pi_1^b}\pi\ (\Box_V\pi \rightarrow \Box_U\pi)$}%
}}}
\put(7351,-1261){\makebox(0,0)[lb]{\smash{\SetFigFont{12}{14.4}{\rmdefault}{\mddefault}{\updefault}{\color[rgb]{0,0,0}$\forall n \ U\vdash \restcons{V}{n}$}%
}}}
\put(5401,-1261){\makebox(0,0)[b]{\smash{\SetFigFont{10}{12.0}{\rmdefault}{\mddefault}{\updefault}{\color[rgb]{0,0,0}$\ref{lemm:OH1}$}%
}}}
\put(5401,-2011){\makebox(0,0)[b]{\smash{\SetFigFont{10}{12.0}{\rmdefault}{\mddefault}{\updefault}{\color[rgb]{0,0,0}$\ref{lemm:OH2}$}%
}}}
\put(6451,-4186){\makebox(0,0)[b]{\smash{\SetFigFont{10}{12.0}{\rmdefault}{\mddefault}{\updefault}{\color[rgb]{0,0,0}$\ref{lemm:OH3}$}%
}}}
\put(4351,-4186){\makebox(0,0)[b]{\smash{\SetFigFont{10}{12.0}{\rmdefault}{\mddefault}{\updefault}{\color[rgb]{0,0,0}$\ref{lemm:OH5}$}%
}}}
\put(7501,-4186){\makebox(0,0)[b]{\smash{\SetFigFont{10}{12.0}{\rmdefault}{\mddefault}{\updefault}{\color[rgb]{0,0,0}$\ref{lemm:OH4}$}%
}}}
\put(3301,-4186){\makebox(0,0)[b]{\smash{\SetFigFont{10}{12.0}{\rmdefault}{\mddefault}{\updefault}{\color[rgb]{0,0,0}$\ref{lemm:OH6}$}%
}}}
\end{picture}

%% file: prooftransformations.pstex_t
\begin{picture}(0,0)%
\includegraphics{prooftransformations.pstex}%
\end{picture}%
\setlength{\unitlength}{2763sp}%
\begingroup\makeatletter\ifx\SetFigFont\undefined%
\gdef\SetFigFont#1#2#3#4#5{%
  \reset@font\fontsize{#1}{#2pt}%
  \fontfamily{#3}\fontseries{#4}\fontshape{#5}%
  \selectfont}%
\fi\endgroup%
\begin{picture}(7737,3766)(526,-3710)
\put(3601,-1861){\makebox(0,0)[lb]{\smash{\SetFigFont{8}{9.6}{\rmdefault}{\mddefault}{\updefault}{\color[rgb]{0,0,0}$\translated{v_1}{k}\ \ \ \ \ldots \ \ \ \ \translated{v_m}{k}$}%
}}}
\put(4201,-3661){\makebox(0,0)[lb]{\smash{\SetFigFont{8}{9.6}{\rmdefault}{\mddefault}{\updefault}{\color[rgb]{0,0,0}$\bot$}%
}}}
\put(6601,-1861){\makebox(0,0)[lb]{\smash{\SetFigFont{8}{9.6}{\rmdefault}{\mddefault}{\updefault}{\color[rgb]{0,0,0}$\translated{v_1}{k}\ \ \ \ \ldots \ \ \ \ \translated{v_m}{k}$}%
}}}
\put(7201,-3661){\makebox(0,0)[lb]{\smash{\SetFigFont{8}{9.6}{\rmdefault}{\mddefault}{\updefault}{\color[rgb]{0,0,0}$\bot$}%
}}}
\put(7576,-886){\makebox(0,0)[lb]{\smash{\SetFigFont{7}{8.4}{\rmdefault}{\mddefault}{\updefault}{\color[rgb]{0,0,0}$u_1\ldots \ u_{l_m}$}%
}}}
\put(676,-136){\makebox(0,0)[lb]{\smash{\SetFigFont{12}{14.4}{\rmdefault}{\mddefault}{\updefault}{\color[rgb]{0,0,0}In $U$:}%
}}}
\put(7126,-1336){\makebox(0,0)[lb]{\smash{\SetFigFont{12}{14.4}{\rmdefault}{\mddefault}{\updefault}{\color[rgb]{0,0,0}$\ldots$}%
}}}
\put(1201,-3661){\makebox(0,0)[lb]{\smash{\SetFigFont{8}{9.6}{\rmdefault}{\mddefault}{\updefault}{\color[rgb]{0,0,0}$\bot$}%
}}}
\put(526,-1861){\makebox(0,0)[lb]{\smash{\SetFigFont{8}{9.6}{\rmdefault}{\mddefault}{\updefault}{\color[rgb]{0,0,0}$v_1\ \ \ \ \ \ \ldots \ \ \ \ \ \ v_m$}%
}}}
\put(2551,-2686){\makebox(0,0)[lb]{\smash{\SetFigFont{8}{9.6}{\rmdefault}{\mddefault}{\updefault}{\color[rgb]{0,0,0}$k$}%
}}}
\put(6151,-886){\makebox(0,0)[lb]{\smash{\SetFigFont{7}{8.4}{\rmdefault}{\mddefault}{\updefault}{\color[rgb]{0,0,0}$u_1\ \ldots \ u_{l_1}$}%
}}}
\put(7876,-1111){\makebox(0,0)[lb]{\smash{\SetFigFont{8}{9.6}{\rmdefault}{\mddefault}{\updefault}{\color[rgb]{0,0,0}$p_m$}%
}}}
\put(6451,-1111){\makebox(0,0)[lb]{\smash{\SetFigFont{8}{9.6}{\rmdefault}{\mddefault}{\updefault}{\color[rgb]{0,0,0}$p_1$}%
}}}
\end{picture}